\documentclass[article,11pt,leqno]{article}
\usepackage{articlemacro}

\DeclareMathOperator{\DM}{DM}

\DeclareMathOperator{\DTM}{DTM}
\newcommand{\quot}[1]{\left[#1\right]}
\newcommand{\Bres}{\textup{Bar}}
\newcommand{\ft}{\textup{ft}}
\newcommand{\B}[1]{\textup{B}#1}

\newcommand{\SH}{\textup{SH}}
\newcommand{\KH}{\textup{KH}}
\newcommand{\GZip}{G\textup{-Zip}}

\newcommand{\KGL}{\textup{KGL}}
\newcommand{\DK}{\textup{DK}}
\date{}
\title{Motivic homotopy theory of the classifying stack of finite groups of Lie type}
\author{Can Yaylali \\ \href{mailto:yaylali@mathematik.tu-darmstadt.de}{yaylali@mathematik.tu-darmstadt.de}}

\begin{document}
\maketitle
\begin{abstract}
Let $G$ be a reductive group over $\mathbb{F}_{p}$ with associated finite group of Lie type $G^{F}$. Let $T$ be a maximal torus contained inside a Borel $B$ of $G$. We relate the (rational) Tate motives of $\B{G^{F}}$ with the $T$-equivariant Tate motives of the flag variety $G/B$.\par On the way, we show that for a reductive group $G$ over a field $k$, with maximal Torus $T$ and absolute Weyl group $W$, acting on a smooth finite type $k$-scheme $X$, we have an isomorphism $A^{n}_{G}(X,m)_{\mathbb{Q}}\cong A^{n}_{T}(X,m)_{\mathbb{Q}}^{W}$ extending the classical result of Edidin-Graham to higher equivariant Chow groups in the non-split case.\par
We also extend our main result to reductive group schemes over a regular base that admit maximal tori. Further, we apply our methods to more general quotient stacks. In this way, we are able to compute the motive of the stack of $G$-zips introduced by Pink-Wedhorn-Ziegler for reductive groups over fields of positive characteristic. 
\end{abstract}
\tableofcontents

\section{Introduction}
\numberwithin{equation}{section}
Let $G$ be a reductive group over $\FF_{q}$, a finite field of characteristic $p>0$, and $\varphi\colon G\rightarrow G$ the $q$-Frobenius. Then $G$ acts on itself via $\varphi$-conjugation, i.e. $(g,h)\mapsto gh\varphi(g)^{-1}$. The stabilizer of the neutral element is denoted by $G^{F}$. If $\overline{\FF}_{q}$ denotes an algebraic closure of $\FF_{q}$, then $G^{F}(\overline{\FF}_{q})=G(\FF_{q})$ is a finite group of Lie-type. The representation of finite groups of Lie type over fields of characteristic $0$ was studied by Deligne and Lusztig (cf. \cite{DL}). In their article they construct representations of $G(\FF_{q})$ by an action on the $\ell$-adic cohomology of certain varieties, for $\ell\neq p$. Roughly, the varieties in question are constructed by intersection of Bruhat strata and graph of Frobenius.\par Let us fix a Borel $B$ of $G$.\footnote{Any reductive group over a finite field is quasi-split and thus admits a Borel.} The Bruhat strata of $\quot{G/B}$ are induced by the Bruhat decomposition of $G$ via pullback along $G/B\rightarrow G/B\times^{G}G/B\cong \quot{B\bs G/B}$. In this article, we want to analyze the cohomological connection between $\B{G^{F}}$ and $\quot{B\bs G/B}$, i.e. study how their motivic categories are related.\par
The derived category of $\ell$-adic sheaves $D(\B{G^{F}},\QQ_{\ell})$, for $\ell\neq p$, encodes information about the action of $G^{F}$ on $\ell$-adic cohomology. One can show that $\B{G^{F}}\cong \quot{G/_{\varphi}G}$, where $G$ acts on itself via $\varphi$-conjugation. The restriction of the $\varphi$-conjugation to $B$ yields an adjunction
$$
 \begin{tikzcd}
	D(\quot{G/_{\varphi}G},\QQ_{\ell})\arrow[r,"",shift left = 0.3em]&\arrow[l,"",shift left = 0.3em]D(\quot{G/_{\varphi}B},\QQ_{\ell}).
\end{tikzcd}
$$ 
 On the other hand there is an adjunction 
$$
 \begin{tikzcd}
	D(\quot{G/_{\varphi}B},\QQ_{\ell})\arrow[r,"",shift left = 0.3em]&\arrow[l,"",shift left = 0.3em]D_{B}(G/B,\QQ_{\ell}),
\end{tikzcd}
$$ 
which is induced via the graph of $\varphi$ (here the right hand side denotes the derived category of $B$-equivariant $\QQ_{\ell}$-modules on the flag variety). Thus, it seems natural that the study of these two adjunctions should lead to information about the geometric representation theory of $G^{F}$ and connection to the classical theory of Deligne-Lusztig.\par
Instead of rewriting the theory of Deligne-Lusztig in the derived setting, we want to understand the adjunctions above in the motivic setting with rational coefficients. The idea is that first after $\ell$-adic realization, we get the classical situation back and further this could lead to information about the $\QQ$-representations of $G^{F}$, as we are naturally working with rational coefficients. 

\subsection*{Motives and connection to representation theory}
Motives were famously envisioned by Grothendieck to capture similar behavior of cohomology theories in an abelian category. The construction of such a category is not an easy task and has been studied for many years. The main approach is to define a derived category of motives with the hope to find a $t$-structure on it, so that the heart of this $t$-structure defines the abelian category of motives. To capture functorial behavior on cohomology theories one demands a full six functor formalism for the derived category of motives. There are several versions of the derived category of motives which agree under certain assumptions. One version was constructed by Cisinki and D\'eglise in the case of rational coefficients, which we denote by $\DM$ (cf. \cite{CD1}). They show that the assignment $X\mapsto \DM(X)$ from smooth $k$-schemes indeed admits a six functor formalism $(\otimes \dashv \Homline,f^{*}\dashv f_{*},f_{!}\dashv f^{!})$ and agrees with the classical construction of Morel. In particular, they show that motivic cohomology, i.e. $\Hom_{\DM(X)}(1_{X},1_{X}(n)[m])$ agrees with Bloch's higher Chow groups $A^{n}(X,2n-m)_{\QQ}$. With the help of the $6$-functor formalism, we can define the motive of an $k$-scheme $\pi\colon X\rightarrow \Spec(k)$ resp. the global sections via $$M_{k}(X)\coloneqq \pi_{!}\pi^{!}1_{Y}\quad \textup{resp.}\quad R\Gamma_{k}(X,\QQ)\coloneqq \pi_{*}\pi^{*}1_{k}$$ computing motivic cohomology resp. cohomology. The existence of a $t$-structure is a more delicate problem and in general not known.  Levine shows that for a particular class of schemes $X$, e.g. finite fields or affine spaces over finite fields, a $t$-structure exists on the full triangulated subcategory of Tate motives $\DTM(X)\subseteq \DM(X)$ generated by the $1_{X}(n)$ for $n\in \ZZ$ (cf. \cite{Levine}). Further, using weight structures one can see that $\DTM(\FF_{p})$ is equivalent to the bounded derived category of $\QQ$-vector spaces. We also have realization functors $\textup{real}_{\ell}\colon \DTM(X)\rightarrow D_{\et}(X,\QQ_{\ell})$ for $\ell\neq p$ that is conservative and $t$-exact for the perverse $t$-structure on $D_{\et}(X,\QQ_{\ell})$. Let us remark that if for a morphism $f\colon X\rightarrow Y$ we have $f_{*}1_{X}\in \DTM(Y)$, then this automatically induces an adjunction of $\DTM(X)$ and $\DTM(Y)$. In particular, the adjunction $f^{*}\dashv f_{*}$ restricts to Tate motives. We call morphisms with such a property \textit{Tate}.\par 
Let us now explain the relation between Tate motives and geometric representation theory. For simplicity, let us assume that $G/\FF_{q}$ is a split reductive group. Let $T$ be a split maximal torus inside a Borel $B$ of $G$. In \cite{SW}, Soergel and Wendt show that for schemes stratified by affine spaces, such as the flag variety $G/B$, one can define a subcategory of $\DM$ called \textit{stratified Tate motives} that admits a $t$-structure. This $t$-structure is glued from the $t$-structure on the strata. For $G/B$ with the Bruhat stratification, we will denote the category of stratified Tate motives with $\DTM_{(B)}(G/B)$. Soergel and Wendt show that $\DTM_{(B)}(G/B)$ is equivalent to the bounded derived category $D^{b}(\Ocal^{\ZZ,ev})$ of graded $\Ocal\coloneqq H^{*}(G/B)$-modules concentrated in even degrees. To connect this to representations, we have to go further and endow motivic cohomology with group actions.\par
For this, we need to define equivariant motives. To make sense of the following construction, we need to work in the setting of $\infty$-categories. The idea is to define $\DM(\Xcal)$ for an Artin stack $\Xcal$ via gluing along the atlas. As we essentially have to glue the derived category this only makes sense in the $\infty$-categorical framework. Then this gluing can be defined via right Kan-extension from schemes to Artin stacks (cf. \cite{RS1}). As one would expect motivic cohomology of a quotient stack $[X/G]$, where $X$ is a smooth $k$-scheme and $G$ is a linear algebraic group, yields the equivariant Chow groups $A^{n}_{G}(X,2n-m)$ of Edidin and Graham (cf. \cite{RS1}). In this way, one can also extend the stratified Tate motives to Artin stacks. In the case of the flag variety, Soergel, Virk and Wendt show that $\DTM(\quot{B\bs G/B})$ is equivalent to the bounded derived category of bi-Soergel-modules (cf. \cite{SVW}). Further, they show that applying $K_{0}$ yields an isomorphism to the Iwahori-Hecke algebra and Verdier duality yields the Kazhdan-Lusztig involution. In particular, the stratified Tate motives of $\quot{B\bs G/B}$ with the weight structure and $6$-functor formalism carries information about the $\ell$-adic geometric representations of $G$.

\subsection*{Connection between finite groups of Lie type and their associated flag variety}
As we have seen above the geometric representation theory of the flag variety is linked to stratified Tate motives. This particular connection uses that the flag variety is stratified by affine spaces and that Tate motives behave nicely under this stratification. We expect that the geometric representation theory of $G^{F}$ is linked to Tate motives on $\B{G^{F}}$, the classifying space of $G^{F}$. In this way, we want to relate representation theory of $G^{F}$ with the representation theory of the associated flag variety.\par 
We will do this by linking the Tate motives of $\B{G^{F}}$ to Tate motives on $\quot{B\bs G/B}$ in the following way. The stack $\B{G^{F}}$ is equivalent to $\quot{G/_{\varphi}G}$, where $G$ acts on itself via $\varphi$-conjugation. Let us fix a maximal torus $T\subseteq B$. Let $T$ act on $G$ also via $\varphi$-conjugation. We can embed $T$ into $T\times T$ via $t\mapsto (t,\varphi(t))$. If we now let $T\times T$ act on $G$ via $(t,t',g)\mapsto tgt'^{-1}$, we get a zigzag of Artin stacks
\begin{equation}\tag{1}
\label{intro.q}
\begin{tikzcd}
	\quot{G/_{\varphi}G}&\arrow[l,"a",swap] \quot{G/_{\varphi} T}\arrow[r,"b"]&\quot{G/T\times T}.
\end{tikzcd}
\end{equation}

We will see that we naturally obtain a fully faithful functor $\DM(\quot{G/T\times T})\hookrightarrow \DM(\quot{B\bs G/B})$. Thus, on the level of motivic categories, this zigzag yields adjunctions between $\DM(\B{G^{F}})$ and $\DM(\quot{B\bs G/B})$. Now we can formulate the leading question of this article: 
\begin{equation}\tag{$\ast$}
\label{main-question}
	\textit{Do these adjunctions preserve Tate motives?}
\end{equation}
 If $T$ is split, then the answer to this question is positive for $b$. The non-split case will be discussed later, as we need a modification of the property ``Tate''. For $a$ the question has to be slightly modified. While $a$ does not preserve Tate-motives, the difference is only the action by the Weyl group. In particular, the underlying motive of $\quot{G/_{\varphi}G}$ is equivalent to the Weyl-invariant part of $R\Gamma_{S}(\quot{G/_{\varphi} T},\QQ)$. As taking invariants under a finite group preserves Tate-motives, this gives a sufficient approximation to our question ($*$). This yields a first point to access motivic representation theory of $G^{F}$ via the motivic geometric representation theory of $G$.

\subsection*{Equivariant motivic homotopy theory of reductive groups}
From now on let $\Btilde$ be a an excellent noetherian scheme of dimension $\leq 2$ and $S$ a regular connected $\Btilde$-scheme of finite type. We will first work with reductive $S$-group schemes $H$ that admit split maximal tori $Z$. This is not a classical notion and any such reductive group scheme may not be split or even quasi-split (in sense of \cite{SGA3} or \cite{ConradRed}). But in this case, one has that the Weyl-group scheme $N_{H}(Z)/Z$ is represented by a constant group scheme associated to a finite constant group $W_{H}$, which we will call the \textit{Weyl group of $Z$ in $H$} (cf. Section \ref{sec.weyl}). \par
As $S$ is regular noetherian, any maximal torus splits after passage to a Galois cover and we deduce the answer to our leading question (\ref{main-question}) from the split case. \par
From now on let $G$ be a reductive $S$-group scheme that admits a maximal torus $T$. Our main question is about the behavior of Tate motives under the induced maps on $\DM$ corresponding to the zigzag (\ref{intro.q}). We will work in a more general setting and look at $a$ and $b$ separately. \par 

\subsubsection*{Equivariant motives and passage to Tori}

The morphism $a$ resembles the motivic version of a more classical problem on classical problem on Chow groups. Let $X$ be an $S$-scheme locally of finite type with $G$-action, what is the relation between $A^{\bullet}_{G}(X)$ and $A^{\bullet}_{T}(X)$? In \cite{EG} Edidin and Graham answer this question for rational Chow groups in the case where $S=\Spec(k)$ is the spectrum of a field and $T$ is split, i.e. $A^{\bullet}_{G}(X)_{\QQ}\cong A^{\bullet}_{T}(X)^{W}_{\QQ}$, where $W$ denotes the Weyl group of $T$ in $G$.\par 
This isomorphism is just a shadow of an equivalence that can be seen motivically.

\begin{intro-theorem}[\protect{\ref{thm.motive.BG}}]
\label{intro.thm.1}
Let $G$ be a reductive $S$-group scheme with split maximal torus $T$ and Weyl group $W$. Assume $G$ acts on an $S$-scheme $X$ locally of finite type. Then we have
	$$
		R\Gamma_{S}(\quot{X/G},\QQ)\simeq R\Gamma_{S}(\quot{X/T},\QQ)^{W}.
	$$
	In particular, if $R\Gamma_{S}(\quot{X/T},\QQ)$ is Tate, then so is $R\Gamma_{S}(\quot{X/G},\QQ)$.
\end{intro-theorem}

The idea of the proof of Theorem \ref{intro.thm.1} is to factorize $\quot{X/T}\rightarrow \quot{X/G}$ into $$\quot{X/T}\rightarrow \quot{X/N_{G}(T)}\rightarrow \quot{X/G}.$$ Then the first map of the factorization is naturally a $W$-torsor and the second map a $G/N_{G}(T)$-bundle. For torsors under finite groups \'etale descent relates motives via $W$-invariants. For $G/N_{G}(T)$-bundles it suffices to see that $R\Gamma_{S}(G/N_{G}(T),\QQ)$ is trivial. We will prove this by reducing the triviality to the equivalence of the map $K_{0}(S)_{\QQ}\rightarrow K_{0}(G/B)^{W}_{\QQ}$, which after comparison with rational Chow theory is classical.  \par 
Applying Theorem \ref{intro.thm.1} to motivic cohomology in the case, where $S=\Spec(k)$, we can extend the classical result to \textit{higher} Chow groups even in the \textit{non-split case}, generalizing the analogous result from \cite{KriC} from split reductive groups to non-split groups. The idea is that any maximal torus splits, after passage to a finite Galois extension. Then we reduce to the split case, using that motives of torsors under finite groups are related by taking invariants.

\begin{intro-corollary}[\protect{\ref{cor.chow.bg}}]
	Let $k$ be a field an let $G$ be a reductive $k$-group scheme with maximal torus $T$ and absolute Weyl group $W$. Assume $G$ acts on a smooth finite type $k$-scheme $X$. Then for all $n,m\in \ZZ$, we have 
	$$
		A^{n}_{G}(X,m)_{\QQ}\cong A^{n}_{T}(X,m)_{\QQ}^{W}.
	$$
\end{intro-corollary}

\subsubsection*{Motives of $T$-torsors}
We will now analyze morphism $b$. For this let $\varphi\colon G\rightarrow G$ be an isogeny that fixes $T$. Consider the embedding $T\hookrightarrow T\times T$ given by the graph of $\varphi$. The quotient $T\times T/T$ under this embedding is isomorphic to $T$. In particular, this isomorphism gives the map 
$$
	b\colon \quot{G/_{\varphi}T}\rightarrow \quot{G/T\times T}
$$
the structure of a $T$-torsor. So, to understand the motivic behavior of $b$ it is enough to understand motives of $T$-torsors.\par 
Thus, let $X\rightarrow Y$ be a morphism of Artin stacks that is a $T$-torsor. Classically, Chow groups in this setting can be computed rather easily. For each character $\chi\in T\rightarrow \GG_{\textup{m}}$ we get a $1$-dimensional representation $\kappa(\chi)$ of $T$. This yields a line bundle $L_{\chi}\coloneqq X\times^{T}\kappa(\chi)$ on $Y$. Multiplication with the first Chern class of $L_{\chi}$ yields an action of the character group $\That$ of $T$ on $A^{\bullet}(Y)$. If $T$ is split, the morphism $b$ yields
\begin{equation}\tag{2}
	\label{intro.eq.1}
	A^{\bullet}_{T}(G)\cong A^{\bullet}_{T\times T}(G)/\That A^{\bullet}_{T\times T}(G)
\end{equation}
	
(cf. \cite{TotaroG}).\par 
Again, this is just a shadow of computations for oriented cohomology theories.\par 
The idea is the following. Assume that $T$ is a split torus. Then we have $T\cong \GG_{\textup{m}}^{r}$ for some $r\in\NN$. By applying successive $\GG_{\textup{m}}$-quotients, we can write 
$$
	X\rightarrow X_{1}\rightarrow X_{2}\rightarrow\dots\rightarrow X_{r}\cong Y,
$$
where $X_{i} \coloneqq \quot{X/\GG_{\textup{m}}^{i}}$. Each of the maps $X_{i-1}\rightarrow X_{i}$ is a $\GG_{\textup{m}}$-torsor. So we may reduce to the case, where $T=\GG_{\textup{m}}$. In this case, we can follow \cite{HosLeh} and assign the line bundle $\Lcal\coloneqq X\times^{\GG_{\textup{m}}} \AA^{1}$ over $Y$. Multiplication with the first Chern class of $\Lcal$ yields a fibre sequence
$$
	M_{k}(X)\rightarrow M_{k}(Y)\rightarrow M_{k}(Y)(1)[2].
$$

Applying this construction in the split case to motivic cohomology yields a generalization of (\ref{intro.eq.1}).

\begin{intro-corollary}[\protect{\ref{cor.T-tors.coho}}]
	Assume that $S=\Spec(k)$. Let $X\rightarrow Y$ be a $T$-torsor of smooth Artin stacks over $S$. Then 
	$$
		A^{\bullet}(X)_{\QQ}\cong A^{\bullet}(Y)_{\QQ}/\That A^{\bullet}(Y)_{\QQ}.
	$$
\end{intro-corollary}

If $X$ and $Y$ are represented by quotients of qcqs smooth schemes by diagonalizable group schemes, for example for $b$ as above, we can actually replace the Chow ring with equivariant $K_{0}$. More generally, the analogous statement holds for any oriented cohomology theory that is $m$-connective (cf. Remark \ref{rem.scalloped}).

\begin{intro-proposition}[\protect{\ref{cor.T-tors.tate}}]
\label{intro.prop.1}
Let $f\colon X\rightarrow Y$ be a $T$-torsor of smooth Artin stacks over $S$. Then $f$ is Tate.
\end{intro-proposition}

If $T$ is not split, we have to slightly modify our notion of Tate motives. We have to add motives by finite field extensions. The resulting motives are called \textit{Artin-Tate} motives. In this way it is clear that if $T$ is not split, then a $T$-torsor is Artin-Tate. We will discuss this in Remark \ref{rem.non.split}.

\subsubsection*{Applications to quotients up to conjugation by isogeny}

We will summarize our results above and apply them to our starting question. We want to remark that we have not particularly used the fact that we are interested in Frobenius-conjugation but rather conjugation up to \textit{isogeny}. Thus, we can work in a more general setting, that we describe in the following.\par
Assume that $S=\Spec(k)$ is the spectrum of a field. Let $P$ resp. $Q$ be parabolics inside $G$ with Levi-components $L$ resp. $M$. Let $\varphi\colon L\rightarrow M$ be an isogeny. Then $L$ acts on $G$ via $(l,g)\mapsto lg\varphi(l)^{-1}$. Let $T$ be a split maximal torus of $G$ contained in $L$. Fix a $g_{0}\in G(S)$ such that $g_{0}\varphi(T)g_{0}^{-1} =T$ and denote by $\widetilde{\varphi}$ the composition of $\varphi$ and $g_{0}$-conjugation. We can embed $T$ into $T\times T$ via $t\mapsto (t,\widetilde{\varphi}(t))$.\par 
We will assume for simplicity that $T$ is split.

\begin{intro-theorem}[\protect{\ref{thm.main}}]
\label{intro.thm.2}
The motives $R\Gamma_{S}(\quot{G/_{\varphi}L},\QQ)$ and $M_{S}(\quot{G/_{\varphi}L})$ are completed Tate motives in $\DM(S)$, i.e. contained in the cocomplete stable subcategory generated by Tate motives. We can compute the Chow ring of $\quot{G/_{\varphi}L}$ as
	$$
		A^{\bullet}(\quot{G/_{\varphi}L})_{\QQ}\cong \left( A^{\bullet}_{T}(G/T)_{\QQ}/\That A^{\bullet}_{T}(G/T)_{\QQ}\right)^{W_{L}},
	$$
	where $W_{L}$ denotes the Weyl group of $T$ inside $L$.
\end{intro-theorem}
If $T$ is not split, we again have to include motives coming from finite field extensions of $k$. Further, we have to replace the Weyl group, by the absolute Weyl group.\par

Using results about equivariant $K$-theory by Uma and Krishna and our results on the cohomology theory of $T$-torsors, we can extend the situation above to equivariant $K$-theory and generalize a result by Brokemper.

\begin{intro-corollary}[\protect{\ref{prop.chev}}]
	In the setting above assume that $G$ is split with respect to $T$ and that the derived group of $G$ is simply connected. Then we have
	$$
		K_{0}(\quot{G/_{\varphi}L})_{\QQ} \cong R(T)_{\QQ}^{W_{L}}/(f-\varphitilde f\mid f\in R(T)_{\QQ}^{W}),
	$$
	where $W_{L}$ denotes the Weyl group of $T$ in $L$ and $W$ as before denotes the Weyl group of $T$ in $G$.
\end{intro-corollary}

\begin{intro-example} 
\label{intro.ex.1}
Let us give two interesting examples, where Theorem \ref{intro.thm.2} can be used.
	\begin{enumerate}
		\item  Let $k=\FF_{q}$ be a finite field with $q$-elements and assume $S=\Spec(k)$. If we set $L=G$ and $\varphi$ the $q$-Frobenius, then we precisely get the situation of the beginning back. In particular, we see that there is an adjunction between $\DTM(\B{G^{F}})$ and $\DTM(\quot{B\bs G/B})$.
		\item Let $k$ be a finite field of characteristic $p>0$ and assume $S=\Spec(k)$.  Another interesting example is the stack of $G$-zips of Pink-Wedhorn-Ziegler (cf. Example \ref{ex.gzip.tate}). In particular, we can recover the computations of Brokemper \cite{Bro1} for Chow groups (up to some computations of \textit{loc.cit.}).
	\end{enumerate}
\end{intro-example}

\subsubsection*{Structure of this article}

We start this article by recalling properties of the $\infty$-category of motives and how to extend this to arbitrary locally of finite type Artin stacks. After defining the necessary notions for this article, we quickly recollect some computational aspects.\par
Afterwards, we start to focus on motives on schemes with group action. First, we explain how to achieve a group action on motives and how torsors under finite groups of Artin stacks have a particular behavior. Then we concentrate on the case $T\subseteq G$, a split maximal torus inside a reductive group. Namely, we show that the relation between $T$-equivariant Chow groups and $G$-equivariant Chow groups extend to the motivic case.\par 
Next, we show that $T$-torsors of Artin stacks are Tate and explicitly compute motivic cohomology implying the classical case of Chow groups.\par 
In the end, we focus on reductive groups with conjugation up to isogeny. We use our results from before to get the desired adjunction of Tate motives with the $T$-equivariant flag variety. We end the paper with ideas for generalization, that we want to address in the future.

\subsubsection*{Setup}
\begin{enumerate}
\item[$\bullet$] Throughout, we fix a noetherian excellent scheme $\Btilde$ of dimension at most $2$ and a regular scheme $S$ of finite type over $\Btilde$.
\item[$\bullet$] An Artin stack is an algebraic stack in the sense of \cite{stacks-project}. When we write ``stack'', we always mean ``Artin stack''. Every Artin stack (and hence scheme) will be locally of finite type over $S$ and any morphism of Artin stacks will be an $S$-morphism.
\item[$\bullet$] Throughout, we will work in the setting of $\infty$-categories and freely use the language of $\infty$-categories. In particular, presheaves will be always considered as presheaves with values in $\infty$-groupoids.
\item[$\bullet$] Throughout $\DM$ denotes the Beilinson motives with coefficients in $\QQ$.
\item[$\bullet$] We will work with reductive group schemes over $S$ (cf. \cite{SGA3}), i.e. $S$-affine smooth group schemes such that the geometric fibers are connected reductive groups. 
\item[$\bullet$] We say that a reductive $S$-group scheme $G$ is \textit{quasi-split} if it so in the sense of \cite[Exp. 24 \S 3.9]{SGA3}.  
\item[$\bullet$] Let $X$ be an $S$-scheme and $G$ an $S$-group scheme acting on $X$. Then we denote by $\quot{X/G}$ the associated quotient stack in the \'etale topology.
\item[$\bullet$] For a $S$-group scheme $G$, we denote its (\'etale) classifying stack by $\B{G}\coloneqq \quot{S/G}$, where $G$ acts trivially on $S$.
\end{enumerate}

\subsection*{Acknowledgement}
 I would like to thank Torsten Wedhorn for multiple discussions and comments on the earlier version, as well as his communication of Appendix \ref{sec.app} and the idea to use reductive group schemes with split maximal tori, generalizing the earlier versions.  Further, I would like to thank Paul Ziegler, who communicated this project and shared his thoughts with me. Also, I would like to thank Thibaud van den Hove for noticing an error in the previous version and Rizacan Ciloguli for lots of very nice discussions. Finally, I would like to thank Arnaud \'Eteve, Tim Holzschuh, Marc Hoyois, Adeel Khan, Timo Richarz, Jakob Scholbach, Fabio Tanania for fruitful discussions and feedback. \par
This project was funded by the Deutsche Forschungsgemeinschaft (DFG, German Research Foundation) - project number 524431573, the DFG TRR 326 \textit{Geometry and Arithmetic of Uniformized Structures}, project number 444845124 and by the LOEWE grant `Uniformized Structures in Algebra and Geometry'.
\section{Rational equivariant motivic homotopy theory}

In this section, we want to recall some properties of the category of (rational) motives and how to extend this to Artin stacks locally of finite type. We expect that most readers are familiar with the notion of motives and the $6$-functor formalism and refer to \cite[Syn. 2.1.1]{RS1} for an overview of the properties of the $6$-functor formalism.\par 
Nevertheless, to prevent confusion, let us quickly recall some common notation and remarks. 
\begin{rem}[\protect{\cite{RS1}}, \protect{\cite{CD1}}]
\label{rem.6.ff}
	In the following any scheme and any morphism will be considered in the category of finite type $S$-schemes, $\Sch_{S}^{\ft}$.
\begin{enumerate}
	\item[(i)]  For any $S$-scheme $X$, $\DM(X)$ is a stable, presentable, closed symmetric monoidal $\infty$-category. The $\otimes$-unit will be denoted by $1_{X}$. It has all limits and colimits.
	\item[(ii)] The assignment $X\mapsto \DM(X)$ can be upgraded to a presheaf of symmetric monoidal $\infty$-categories
	$$
		\DM^{*}\colon \Sch_{S}^{\ft}\rightarrow \ICat^{\otimes},\ X\mapsto \DM(X),\ f\mapsto f^{*}.
	$$
	  For any morphism of schemes $f\colon X\rightarrow Y$, there is an adjunction
	$$
		\begin{tikzcd}
			 f^{*}\colon \DM(Y)\arrow[r,"",shift left = 0.3em]&\arrow[l,"",shift left = 0.3em]\DM(X)\colon f_{*}.
		\end{tikzcd}
	$$
	\item[(iii)] If $f$ is smooth, then $f^*$ has a left adjoint, denoted $f_\sharp$.
	\item[(iv)] The assignment $X \mapsto \DM(X)$ can be upgraded to a presheaf of $\infty$-categories
$$\DM^{!}: (\Sch_{S}^{\ft})^{\op} \rightarrow \ICat,\ X \mapsto \DM(X), f \mapsto f^!.$$
For each $f$, there is an adjunction
$$f_! : \DM(X) \rightleftarrows \DM(Y): f^!.$$
For any factorization $f = p \circ j$ with $j$ an open immersion and $p$ a proper map, there is a natural equivalence $f_! \cong p_* j_\sharp$.
	\item[(v)] Both functors $\DM^{*}$ and $\DM^{!}$ are sheaves for the $h$-topology (cf. \cite[Thm. 2.1.13]{RS1})
	\item[(vi)] For the projection $p: \GG_{\textup{m},S}\times_{S} X \rightarrow X$, and any $M \in \DM(X)$, the map induced by the counit $p_\sharp p^* M[-1] \rightarrow M[-1]$ in $\DM(X)$ is a split monomorphism.
The complementary summand is denoted by $M(1)$.
The functor $M \mapsto M(1)$ is an equivalence with inverse denoted by $M \mapsto M(-1)$. For any integer $n$ the $n$-fold composition is denoted by $M\mapsto M(n)$ and in the future, we will abbreviate $\langle n\rangle \coloneqq (n)[2n]$.
\end{enumerate}
\end{rem}
\par
Let $\Xline$ be a prestack\footnote{This notion is from \cite{RS1}.}, i.e. presheaf of anima on the category of rings. There are several approaches to the $\infty$-category $\DM(\Xline)$. If $\Xline$ is an Artin stack over a field $k$ Hoskins-Lehalleur use a construction similar to equivariant Chow groups. One resolves $\Xline$ by open substacks $(\Xline_{i})$ such that on each $\Xline_{i}$ there is a vector bundle $V_{i}$ together with an open $U_{i}$ with a free $G$-action such that the codimension of $V_{i}\setminus U_{i}$ tends towards infinity (cf. \cite{HosLeh}). This idea was further generalized to arbitrary cohomology theories by Khan-Ravi by a construction they call \textit{lisse extension} (cf. \cite{KR1}). This construction was already used for the motive of classifying stacks $\B{G}$ by Morel-Voevodsky (cf. \cite[\S 4]{MV1}). Totaro then gave an explicit computation of the motive of $\B{\GG_{\textup{m}}}$ over a field (cf. \cite{Tot} and Example \ref{ex.BGL}).\par Alternatively to the construction of Hoskins-Lehalleur resp. Khan-Ravi, Richarz-Scholbach give a construction via certain left and right Kan extension (cf. \cite{RS1}). Their approach is based on gluing the motivic structure on Beilinsion motives to arbitrary prestacks. Indeed,  $\DM(-)$ satisfies $h$-descent, so it is rather formal to extend the six functor formalism to Artin stacks locally of finite type, we will use this approach. One should note that this was also discussed in \cite{Khan1}, to extend the $6$-functor formalism to higher Artin stacks. \par 
For computations of the underlying motives it seems to be better to work with the definition of Hoskins-Lehalleur resp. Morel-Voevodsky. Let $f\colon \Xfr\rightarrow \Spec(k)$ be a smooth Artin stack and let $M(\Xfr)$ denote the $k$-linear motive of $\Xfr$ defined in \cite{HosLeh}. This defines an object in $\DM^{*}(\Spec(k))$. Let $1_{k}$ denote the unit in $\DM^{*}(\Spec(k))$. We will see in Corollary \ref{cor.rel.tot} that if $f$ is smooth, we have $M(\Xfr)\simeq f_{\sharp}f^{*}1_{k}$. In particular, if we use the approach of \cite{RS1} and define a motive of a prestack as the $\sharp$-push/$*$-pull of the unit, we see that the notion of motives on Artin stacks as in \textit{loc.cit.} agrees with the classical ones.\par 
To define $\DM$ for prestacks, we will follow \cite{RS1}. For the following definition let us fix a regular cardinal $\kappa$. Let $\textup{Aff}_{S}^{\ft}$ denote the (Nerve of the) category of affine schemes of finite type over $S$. We let $\textup{Aff}_{S}^{\kappa}$ denote the $\kappa$-pro-completion of $\textup{Aff}_{S}^{\ft}$. Further, $\textup{DGCat}_{\textup{cont}}$ denotes the $\infty$-category of presentable stable $\QQ$-linear dg-$\infty$-categories with colimit preserving functors.

\begin{defi}[\protect{\cite{RS1}}]
Let $y\colon \textup{Aff}_{S}^{\kappa}\hookrightarrow P(\textup{Aff}_{S}^{\kappa})$ be the Yoneda embedding. We define the functor
$$
	\DM_{S}\colon P(\textup{Aff}_{S})^{\op}\rightarrow \textup{DGCat}_{\textup{cont}}
$$
via left Kan extension along the inclusion $\textup{Aff}_{S}^{\ft}\subseteq \textup{Aff}_{S}^{\kappa}$ and right Kan extension along $\textup{Aff}_{S}^{\kappa} P(\textup{Aff}_{S}^{\kappa})$, where all the functors are given via $!$-pullback.
For a prestack $\Xline\in P(\textup{Aff}_{S}^{\kappa})$, we define the  $\infty$-category of \textit{$S$-linear motives (with rational coefficients) of $\Xline$} as $\DM_{S}(\Xline)$.
\end{defi}
As noted in \cite{RS1}, for applications we can take $\kappa$ to be large enough, so that all affine schemes of interest are inside $\textup{Aff}_{S}^{\kappa}$. Thus, for once and for all we may fix a regular cardinal $\kappa$ and drop it from the notation.\par

Khan showed in \cite[Thm. A.5]{Khan1} that this method of extending the theory of motives to (derived) Artin stacks, does not loose the $6$-functor formalism. One way to see this, is that we can use the DESCENT program in \cite{DESCENT}, since the Beilinsion motives satisfy \'etale descent, in our context. As mentioned in \cite{Khan1}, this is equivalent to the construction of \cite{RS1}.

\begin{thm}
	Let $\widetilde{\DM}$ be the restriction of $\DM$ to Artin stacks locally of finite type over $S$. Then $\widetilde{\DM}$ is compatible with the $6$-functor formalism in the sense of \cite[Syn. 2.1.1]{RS1}.
\end{thm}
\begin{proof}
	The proof is the same as \cite[Thm A.5]{Khan1}\footnote{As mentioned in \textit{op.cit.}, the method of extending the $6$-functor formalism works with any motivic category that satisfies \'etale descent.} and follows from the results of \cite{RS1}. 
\end{proof}

\begin{defi}[\protect{\cite{RS1}}]
\label{def.motive}
	Let $\Xfr$ be an Artin $S$-stack locally of finite type with structure morphism $f\colon \Xfr\rightarrow S$. Then we define the \textit{(rational) $S$-linear motive of $\Xfr$} as $M_{S}(\Xfr)\coloneqq f_{!}f^{!}1_{S}$. If $S=\Spec(A)$ is affine, we write $M_{A}(\Xfr)$.\par 
	We further define the \textit{global sections of $\Xfr$ over $S$} to be $R\Gamma_{S}(\Xfr,\QQ)\coloneqq f_{*}f^{*}1_{S}$.
\end{defi}

\begin{rem}
	Let $f\colon \Xfr\rightarrow S$ be an Artin stack locally of finite type over $S$. If $f$ is smooth, then relative purity implies that $f_{\sharp}f^{*}\simeq f_{!}f^{!}$ and in particular, we see with
	$$
		\Hom_{\DM(S)}(M_{S}(\Xfr),1(n)[m]) \simeq \Hom_{\DM(S)}(1(-n)[-m],R\Gamma_{S}(\Xfr,\QQ))
	$$
	that $M_{S}(\Xfr)$ computes motivic homology and $R\Gamma_{S}(\Xfr,\QQ)$ motivic cohomology.
\end{rem}

\begin{notation}
	Let $G$ be an $S$-group scheme locally of finite type acting on an $S$-scheme $X$ locally of finite type, via a morphism $a$. For the quotient stack $\quot{X/G}$, we can define a simplicial object in $S$-schemes locally of finite type, via its \textit{Bar-resolution}
$$
	\begin{tikzcd}
		\dots\arrow[r,"", shift left = 0.6 em]\arrow[r,"",shift right= 0.6 em]\arrow[r,""]&G\times_{S} X\arrow[r,"p", shift left = 0.3 em]\arrow[r,"a",shift right= 0.3 em,swap]& X
	\end{tikzcd}
$$
We denote the corresponding simplicial functor with $\Bres^{\bullet}(X,G)$.
\end{notation}

There is also an alternative way to define motives of algebraic stacks via the Bar-resolution. For each $n\geq 0$ let $\QQ(\Bres^{n}(X,G))$ be the constant \'etale sheaf with coefficients in $\QQ$ associated to $\Bres^{n}(X,G)$. This yields a simplicial object in \'etale sheaves of $S$-schemes locally of finite type with rational coefficients. The complex associated to this simplicial object induces a motive $M_{S}(\Bres^{\bullet}(X,G))$ in $\DM(S)$. For $S=\Spec(k)$ the spectrum of a field, Hoskins-Lehalleur explain in \cite{HosLeh} that this definition is equivalent to their definition of a motive of an Artin stack. The naturally arising question is if $M_{S}(\Bres^{\bullet}(X,G))$ is equivalent to $M_{S}(\quot{X/G})$ as defined in Definition \ref{def.motive}. If $\quot{X/G}$ is representable by a smooth scheme, then the answer is positive and follows by cohomological descent for $\DM$ with respect to the $h$-topology (cf. \cite[Thm. 14.3.4, Prop. 5.2.10]{CD1}). Thus, the answer stays positive for smooth Artin stacks, as $\DM$ satisfies $h$-descent by gluing (cf. \cite[Thm. 2.2.16]{RS1}).

\begin{cor}
\label{cor.rel.tot}
	Let $k$ be a field and $S=\Spec(k)$. Let $X$ be a smooth $S$-scheme and $G$ be a smooth $S$-group scheme acting on $X$ with structure map $f\colon \quot{X/G}\rightarrow S$. Then $M_{S}(\quot{X/G})$ is equivalent to $M_{S}(\Bres^{\bullet}(X,G))$.
\end{cor}
\begin{proof}
	This follows from \cite[Prop. A.7]{HosLeh} and the discussion above.
\end{proof}

We can use Corollary \ref{cor.rel.tot} to compute the motive of $B\GG_{\textup{m}}$ as in \cite{Tot}.

\begin{example}
\label{ex.BGL}
	Let $k$ be a field. Further, let $\GG_{\textup{m},k}$ act trivially on $\Spec(k)$. Then $$M_{k}(\B{\GG_{\textup{m},k}})\simeq \colim_{i\in\NN}M_{k}(\PP^{i}_{k})\simeq \bigoplus_{i\geq 0}1_{k}\langle i\rangle.$$
\end{example}

\begin{rem}
	In the following we want to understand the Gysin sequence for algebraic stacks. Let us quickly recall it in the scheme case.\par  Let $i\colon Z\hookrightarrow X$ be a closed immersion of $S$-schemes of pure codimension $n$ with open complement $U$. Let us assume that $Z$ and $X$ are smooth over $S$. In particular, we see that $i$ is equivalently a regular closed immersion of codimension $n$. Then there exists a fibre sequence of the form
	$$
		M_{S}(U)\rightarrow M_{S}(X)\rightarrow M_{S}(Z)\langle n\rangle
	$$
	(cf. \cite[11.3.4]{CD1}).
	\par 
	We are going to replace $S$ by a smooth Artin stack $\Yfr$ over $S$ and $X$ by a smooth Artin stack $\Xfr$ over $\Yfr$. For this let us also recall the notion of a (regular) closed immersion of a certain codimension for Artin stacks.\par
	Let $\iota\colon \Zfr\hookrightarrow \Xfr$ be a closed immersion of locally noetherian Artin stacks. Let $X\rightarrow \Xfr$ be a smooth atlas. Then $\iota$ is representable and we define the codimension of $\Zfr$ as the codimension of $\Zfr\times_{\Xfr} X$ in $X$ (cf. \cite[\S 6]{Osser}). We can also define the notion of a regular immersion in that way (cf. \cite[06FM]{stacks-project}) and the notion of its codimension. In particular, a closed immersion of smooth Artin $S$-stacks $\Zfr\hookrightarrow \Xfr$ is automatically regularly immersed and the codimension of the regular immersion agrees with the codimension as a closed substack.
\end{rem}

\begin{lem}[The Gysin sequence]
\label{Gysin for stacks}
	Let $f\colon \Xfr\rightarrow \Yfr$ be a smooth schematic morphism of smooth Artin $S$-stacks. Further let $i\colon \Zfr\hookrightarrow \Xfr$ be a closed immersion of (pure) codimension $n$ such that $\Zfr$ is smooth over $\Yfr$ with open complement $j\colon \Ufr\rightarrow \Xfr$. Further, let us denote $f_{0}\coloneqq f\circ j$ and $\fbar\coloneqq f\circ i$. Then there exists the following fibre sequence
	$$
		f_{0!}f_{0}^{!}1_{\Yfr}\rightarrow f_{!}f^{!}1_{\Yfr}\rightarrow \fbar_{!}\fbar^{!}1_{\Yfr}\langle n\rangle .
	$$
\end{lem}
\begin{proof}
	Let $Y\rightarrow \Yfr$ be a smooth atlas. Let us define $X\coloneqq Y\times_{\Yfr} \Xfr$ and let $\Cv(Y)_{\bullet}$ resp. $\Cv(X)_{\bullet}$ denote the corresponding \v{C}ech nerves. By construction $\Cv(X)_{\bullet}$ is obtained by $\Cv(Y)_{\bullet}\times_{Y} X$. So, by functoriality we get maps 
	$$
	 \begin{tikzcd}
		f_{\bullet !}\colon \DM(\Cv(Y)_{\bullet})\arrow[r,"",shift left = 0.3em]&\arrow[l,"",shift left = 0.3em]\DM(\Cv(X)_{\bullet})\colon f_{\bullet}^{!}
	\end{tikzcd}
	$$
	 that induce the maps $f_{!}$ and $f^{!}$ after passing to the limit. By construction we have a pullback diagram
	 $$
		\begin{tikzcd}
			\Cv(X)_{\bullet}\arrow[r,"f_{\bullet}"]\arrow[d,"j_{X,\bullet}"]& \Cv(Y)_{\bullet}\arrow[d,"j_{Y,\bullet}"]\\
			\Xfr\arrow[r,""]&\Yfr.
		\end{tikzcd}
	 $$
	 In particular, by smoothness of the atlas and the exchange equivalence, we have 
	 $$
		 j^{*}_{Y,\bullet}f_{!}f^{!}1_{Y}\simeq f_{\bullet !}f_{\bullet}^{!}1_{\Cv(Y)_{\bullet}}.
	 $$ 
	 Thus, by smoothness we can use that $\DM^{*}(\Yfr)\simeq \DM^{!}(\Yfr)$ and descent to see that 
	 $$
	 	\colim_{\Delta} f_{\bullet !}f_{\bullet}^{!}1_{\Cv(Y)_{\bullet}}\simeq f_{!}f^{!}1_{Y}.
	 $$
	 Analogously, we can write  
	 $$
	 	\colim_{\Delta} f_{0\bullet !}f_{0\bullet}^{!}1_{\Cv(Y)_{\bullet}}\simeq f_{0!}f^{!}_{0}1_{Y},\quad \colim_{\Delta} \fbar_{\bullet !}\fbar_{\bullet}^{!}1_{\Cv(Y)_{\bullet}}\simeq \fbar_{!}\fbar^{!}1_{Y}.
	 $$
	 Therefore, we may assume that $\Yfr$ is representable by a scheme and by representability of $f$ also $\Xfr,\Ufr$ and $\Zfr$ are representable by schemes. Hence, the result now follows from the classical Gysin sequence (cf. \cite[11.3.4]{CD1}).
\end{proof}

Lastly, let us define Tate motives. As we mentioned in the introduction, the existence of a motivic $t$-structure is still an open problem. For a field $k$ Levine proved that under certain vanishing assumptions on motivic cohomology in $\DM(k)$, the so called Beilinson-Soul\'e vanishing conjecture, such a $t$-structure exists on the full stable subcategory generated by Tate twists $1_{k}(n)$ (cf. \cite{Levine}). The Beilinson-Soul\'e vanishing conjecture holds for example for finite fields. 

\begin{defi}
	Let $X$ be an Artin $S$-stack locally of finite type. We define the category of Tate motives $\DTM(X)$ to be the full stable subcategory of $\DM(X)$ generated by $1_{X}(n)$, for $n\in\ZZ$. An element $M\in\DM(X)$ is \textit{Tate}, if $M\in\DTM(X)$. \par 
	A map $f\colon X\rightarrow Y$ of Artin stacks locally of finite type over $S$ is called \textit{Tate} if $f_{*}1_{X}$ is Tate.
\end{defi}

Levine further shows, that existence of a weight structure on Tate motives for a field imply that the heart of $\DTM(\FF_{p})$ under the motivic $t$-structure is equivalent to the category of graded finite dimensional $\QQ$-vector spaces $(\QQ\textup{-VS}^{\ZZ})$. In particular, using a result of Wildeshaus, we see for example that $\DTM(\FF_{p})\simeq \Dcal^{b}(\QQ\textup{-VS}^{\ZZ})$, where $\Dcal^{b}$ denotes the bounded derived category (cf. \cite{Wild}).\par 
Let us give a particularly interesting example of a Tate map that will be used later on.

\begin{example}
\label{ex.flag}
	Let $G$ be a split reductive $S$-group scheme and $B\subseteq G$ a Borel. Then we claim that the structure map of the flag variety $G/B\rightarrow S$ is Tate. \par
	Indeed, the Bruhat decomposition of $G/B$ yields a stratification by affine spaces indexed by the Weyl group. The length of each Weyl element yields a partial order on the associated Schubert varieties. With this order, one can show using standard arguments that $R\Gamma_{S}(G/B,\QQ)\simeq \bigoplus_{w\in W} 1_{S}\langle l(w)-n\rangle$, where $n$ denotes the relative dimension of $G/B$ over $S$ (cf. \cite{YayT} or \cite{Bach} for more details on analogous problems). 
\end{example}

\begin{rem}[\protect{\cite{RS1}}]
\label{rem.tate}
	If $f\colon X\rightarrow Y$ is a smooth morphism of Artin $S$-stacks locally of finite type and $X$ is smooth over $S$, then $D_{Y}(f_{*}1_{X}) \simeq f_{!}D_{X}(1_{X}) \simeq f_{!}1_{X}\langle \Omega_{X/S}\rangle$ and $D_{Y}(f_{!}1_{X})\simeq f_{*}1_{X}\langle \Omega_{X/S}\rangle$. Thus, $f_{*}1_{X}$ is Tate if and only if $f_{!}1_{X}$ is Tate.
\end{rem}

In the literature one also considers the stable \textit{cocomplete} $\infty$-category generated by Tate twists (cf. \cite{RS1}). This is usually also referred to as ``Tate motives'' but we will differentiate these from our definition. An example of such motives is given by the motive of $\B{\GG_{\textup{m},k}}\rightarrow \Spec(k)$, the classifying stack of $\GG_{\textup{m},k}$ over a field $k$. Indeed, in Example \ref{ex.BGL} we have seen $M_{k}(\B{\GG_{\textup{m}}})\simeq \bigoplus_{n\in\ZZ} 1_{k}\langle n\rangle$ and thus this lies in the ind-completion of the category of Tate motives. 

\begin{defi}
	Let $X$ be an Artin $S$-stack locally of finite type. We will call an $M\in\DM(X)$ \textit{completed Tate} if it is already in the full stable cocomplete subcategory generated by $1_{X}(n)$.
\end{defi}

\section{Equivariant motives under reductive groups}
\label{sec.beg.eq}
Let us give an outlook of the next subsections in the split case. The non-split case will then be a corollary.\par Assume that $S$ is connected and let $G$ be a reductive $S$-group scheme and $T$ a split maximal torus in $G$ with Weyl group scheme $W_{S}$. We will see in Section \ref{sec.weyl} that $W_{S}$ is a constant group scheme associated to the finite group $W\coloneqq W_{S}(S)$, which we call \textit{Weyl group of $T$ in $G$}.\par 
 Now let $X$ be an $S$-scheme locally of finite type with $G$ action. In this section, we want to show that  $R\Gamma(\quot{X/T},\QQ)^{W}\simeq R\Gamma(\quot{X/G},\QQ)$. In particular, if $R\Gamma(\quot{X/T},\QQ)$ is Tate, then so is $R\Gamma(\quot{X/G},\QQ)$.\par
The key idea is to use the factorization $\quot{X/T}\xrightarrow{g} \quot{X/N}\xrightarrow{h} \quot{X/G}$, where $N$ is the normalizer of $T$ in $G$. We note that per definition the map $g\colon\quot{X/T}\rightarrow \quot{X/N}$ is a $W$-torsor. As $W_{S}$ is finite \'etale, we see that after passage to an \'etale cover, the motive of $\quot{X/N}$ is equivalent to the Weyl invariant part of the motive of $\quot{X/T}$. \par
The map $h\colon\quot{X/N}\rightarrow \quot{X/G}$ is a $G/N$-torsor. For this map, we will show that $h_{*}1\simeq 1$. Let us explain this in a bit more detail. Up to taking $W$-invariant, we can identify $G/N$ with $G/B$. On $G/B$, we have stratification by Schubert cells, which are affine spaces. In this way, we can decompose $p_{*}1_{G/B}$ as a direct sum of twists and shifts indexed by $W$, where $p\colon G/B\rightarrow S$ is the structure map. The equivalence of $p_{*}1_{G/B}^{W}\simeq 1_{G/B}$ reduces to a question on ordinary $K$-theory. More precisely, it will be enough to show that $K_{0}(S)_{\QQ}\rightarrow K(G/T)^{W}_{\QQ}$ is an isomorphism, which after reduction to the case of Chow groups is classically known.\par 
Before coming to our main result of this section, we will first introduce group actions on motives.

\begin{rem}
\label{rem.tors.et}
	A key argument in this section, is that torsors under finite \'etale group schemes have related motives by taking invariants of the action. This is a major obstruction for the generalization to integral coefficients, as we expect that this is only satisfied if we have \'etale descent (c.f. \cite[Thm. 3.3.32]{CD1}). Nevertheless, using the theory of \'etale motives it should not be to difficult to have analogous results after inverting only the residue characteristics of our base scheme.
\end{rem}

\subsection{Preliminaries on reductive group schemes}
\label{sec.weyl}
Let $G$ be a reductive $S$-group scheme. In the following article, we want to work with reductive group schemes that admit maximal tori. Even though this is automatically satisfied if $S$ is the spectrum of a field, this is not true for general base schemes (cf. \cite{Conrad}). We will first proof our results under the assumption that the maximal torus of $G$ is split. This is not equivalent to splitness of $G$ or even quasi-splitness.\par 
In general one can find two main notions of \textit{quasi-split} in the literature. In SGA 3 a quasi-split reductive group scheme admits a Borel pair as well as a global section of the associated Dynkin diagram (cf. \cite[Exp. XXIV]{SGA3}). Another definition is given by Conrad (cf. \cite{ConradRed}), where he only enforces existence of a Borel subgroup. If $S$ is affine, this automatically implies the existence of a maximal torus contained in $B$.\par  As mentioned in the introduction, for applications it is enough to consider only the existence of a maximal torus. If the maximal torus is split, we can also talk about the associated Weyl group, as the next lemma shows.\par 
We want to thank Torsten Wedhorn for explaining this and elaborating that it is enough to consider reductive group schemes with split maximal tori instead of split reductive group schemes.

\begin{lem}
	Assume that $S$ is connected. Let $T$ be a split maximal torus inside $G$ with normalizer $N_{G}(T)$, then $W_{S}\coloneqq N_{G}(T)/T$ is represented by a constant finite group scheme.
\end{lem}
\begin{proof}
	Let us first recall some well known fact (cf. \cite[Exp. XXII \S 3]{SGA3}).\par 
	The quotient $W_{S}$ is represented by a finite \'etale $S$-scheme and that there is an open immersion $\iota\colon W_{S}\hookrightarrow \Autline_{S}(T)$. As $T$ is split, $\Autline_{S}(T)$ is the constant group associated to the automorphism group $\Aut(X^{\bullet}(T))$. In particular, $\Autline_{S}(T) \cong \coprod_{\Aut(X^{\bullet}(T))} S$.\par 
	Note that $\Autline_{S}(T)$ is separated and thus $\iota$ is finite. As any finite monomorphism is a closed immersion (cf. \cite[03BB]{stacks-project}), the morphism $\iota$ is an open and closed immersion. By assumption $S$ is connected and therefore the image of $\iota$ has to be a disjoint union of connected components, i.e. $W_{S}\cong \coprod_{H} S$, for some subset $H\subseteq \Aut(X^{\bullet}(T))$. Finally, since $W_{S}\rightarrow S$ is finite, the set $H$ has to be finite.
\end{proof}

\begin{defi}
	Assume that $S$ is connected and that $G$ admits a split maximal torus. Then we call the finite group $W\coloneqq (N_{G}(T)/T)(S)$ the \textit{Weyl group of $G$}\footnote{If $G$ is split reductive, then this agrees with the classical notion.}.
\end{defi}

Let us give some examples of reductive groups that admit maximal tori.

\begin{example}
\label{ex.red.grp}
Let us give examples of schemes $S$ and $G$ as above.\par 
\begin{enumerate}
	\item[$\bullet$] If $S=\Spec(k)$ is the spectrum of a field, then any reductive group admits a maximal torus (cf. \cite[Exp. XIV Thm 1.1]{SGA3}).
	\item[$\bullet$] If $S$ is any connected regular $\Btilde$-scheme of finite type, then any quasi-split reductive $S$-group scheme $G$ suffices.
	\item[$\bullet$] Assume, that $A$ is a Dedekind domain and $S=\Spec(A)$. Then $G$ admits a maximal torus if its generic fibre quasi-split (cf. Appendix \ref{sec.app}).
	\item[$\bullet$] Brian Conrad constructed examples of reductive groups over $\Spec(\ZZ)$ that do not admit maximal tori (cf. \cite{Conrad}). All of these examples concern non-classical reductive groups and the author is not aware of any other counterexamples.
\end{enumerate}
\end{example}


\subsection{Torsors under finite groups}

\label{sec.fin}

To warm up, let us first look at torsors under finite groups. In particular, it will be clear that the canonical map $\quot{X/T}\rightarrow \quot{X/N}$, considered in the beginning of Section \ref{sec.beg.eq}, yields an equivalence of motives after taking Weyl invariants. In particular, as taking invariants under a finite group is a limit and Tate-motives are closed under finite limits, we can use this for Tateness properties. We will not need any connectedness hypothesis on $S$ for the results of this section.

First, we need to understand group actions on motives and taking fixed points under these actions.\par

An action of a group $\Gamma$ on a motive $M$ in $\DM(X)$ is map $\Gamma\rightarrow \Aut_{\DM(S)}(M)$ that is a group homomorphism on $\pi_{0}$. Or equivalently, it is a map $\Sigma_{\Gamma}\rightarrow \DM(X)$ (here $\Sigma_{\Gamma}$ denotes the deloop of the group $\Gamma$ seen as a discrete category - usually this is denoted with $\B{\Gamma}$ but to avoid confusion, we changed the notation).

\begin{defi}
	Let $M$ be a motive in $\DM(X)$ with an action by a finite group $\Gamma$. Then we define the \textit{homotopy $\Gamma$-fixed points of $M$}, denoted by $M^{h\Gamma}$, as the limit of the action map $\Sigma_{\Gamma}\rightarrow \DM(X)$.
\end{defi}

\begin{defrem}
\label{rem.fix}
	Let $X$ be an Artin $S$-stack locally of finite type and with an action by a finite group $\Gamma$. For each $\gamma$, we have an action map $a_{\gamma}\colon X\rightarrow X$. By construction of $\DM(X)$ this defines a map $\gamma.\colon 1_{X}\rightarrow 1_{X}$ by lax-monoidality of the $*$-pushforward. This endows any $M\in\DM(X)$ with an action via $\gamma.$. We define the \textit{$\Gamma$-fixed points of $M$}, denoted by $M^{\Gamma}$ as the image\footnote{Note that (the homotopy category of) $\DM(X)$ is pseudo-abelian and hence for any idempotent operator we can define its image.} of the map 
	$$
		p = \frac{1}{\# \Gamma}\sum_{\gamma\in \Gamma} \gamma.
	$$
	The canonical map $M^{\Gamma}\rightarrow M$ defines an equivalence $M^{\Gamma}\xrightarrow{\sim}M^{h\Gamma}$ (cf. \cite[3.3.21]{CD1}).\par 
	Let $X\rightarrow Y$ be a map of $\Gamma$-torsor of Artin stacks locally of finite type over $S$ (here we see $\Gamma$ as a constant group scheme on $S$). Then the $\Gamma$-torsor $Y$-automorphisms of $X$ is isomorphic to $\Gamma$ and thus we get a $\Gamma$-action on $f_{*}f^{*}M$ for any $E\in \DM(Y)$ (via $*$-pushforward of a $\Gamma$-torsor $Y$-automorphism). Note that $f_{*}f^{*}E$ can be used to compute the motivic cohomology of $X$ with coefficients in $E$ for smooth $f$ as 
	$$
		\Hom_{\DM(Y)}(1_{Y},f_{*}f^{*}E)= \Hom_{\DM(Y)}(f^{*}1_{Y},f^{*}E) = \Hom_{\DM(Y)}(M_{Y}(X),E).
	$$
	In particular, we see that\footnote{Kernels in triangulated categories are monomorphisms and thus on the level of homotopy groups induce short exact sequences, dualy cokernels induce epimorphisms. Thus, applying the homotopy invariance functor to the Hom-spectrum $\Homline_{\DM(Y)}(1_{Y},f_{*}f^{*}E)$, we see that $\pi_{0}$ precisely gives us the invariants of the induced group action on $\pi_{0}$.}  
	$$
		\Hom_{\DM(Y)}(1_{Y},(f_{*}f^{*}E)^{\Gamma}) = \Hom_{\DM(Y)}(M_{Y}(X),E)^{\Gamma}.
	$$
	Note that as $\Gamma$ is finite, the $\Gamma$-torsor $f$ is \'etale and proper, hence $f_{*}f^{*} \simeq f_{!}f^{!}$.
\end{defrem}

Let $k$ be a field and $X$ a scheme locally of finite type over $\Spec(k)$. Further, let $K/k$ be a finite Galois extension and let us denote the base change of $X$ to $\Spec(K)$ by $X_{K}$. The Chow groups of $X$ and $X_{K}$ are related by the fixed points under the Galois group, i.e. $\textup{CH}_{n}(X) = \textup{CH}_{n}(X_{K})^{\Gal(K/k)}$. As one expects, this also holds motivically. This is due to Ayoub and Cisinki-D\'eglise. As noted in Remark \ref{rem.tors.et}, we do not expect this to hold, when we do not impose \'etale descent. Thus, we do not expect the next lemma to hold with integral coefficients. 

\begin{lem}
\label{lem.torsor.fixed}
	Let $\Gamma$ be a finite group. Let $f\colon X\rightarrow Y$ be a $\Gamma$-torsor of Artin $S$-stacks locally of finite type. Then the unit factors as $\id\rightarrow (f_{*}f^{*})^{\Gamma}\rightarrow f_{*}f^{*}$ and the map $\id\rightarrow (f_{*}f^{*})^{\Gamma}$ is an equivalence.
\end{lem}
\begin{proof}
	The factorization of the unit follows from the description of $(f_{*}f^{*})^{\Gamma}$ as a limit. We claim that $\id\rightarrow (f_{*}f^{*})^{\Gamma}$ is an equivalence. It suffices to check this after base change to a smooth atlas of $Y$. In particular, we may assume that $Y$ is represented by a scheme. Since $\DM_{\QQ}$ satisfies $h$-descent, we may assume\footnote{Using \cite[Prop. 3.3.31]{CD1}, we see that $M_{Y}(X)^{\Gamma}\simeq \varphi_{*}\varphi^{*}1_{Y}$, where $\varphi\colon (\Xscr,\Gamma)\rightarrow Y$ is the induced morphism of the diagram $\Sigma_{\Gamma}\rightarrow \Sch_{S}$ that maps the single point of the category $\Sigma_{\Gamma}$ to $X$ and the morphisms to the actions. Then we can base change using \cite[Prop. 3.1.17]{CD1}.} that $f$ is a trivial $\Gamma$-torsor, where it follows from \cite[Prop. 2.1.166]{ayoub}.
\end{proof}

\begin{example}
	Let us consider the $W$-torsor $f\colon \quot{X/T}\rightarrow \quot{X/N}$  in the beginning of Section \ref{sec.beg.eq}. As $W$ is finite, Lemma \ref{lem.torsor.fixed} implies that $f_{*}f^{*}1_{\quot{X/T}}^{W} \simeq 1_{\quot{X/N}}$. After $\sharp$- resp. $*$-pushforward to the base $S$ along the structure map $\quot{X/N}\rightarrow S$, we see that
	$$
		R\Gamma_{S}(\quot{X/N}\QQ) \simeq R\Gamma_{S}(\quot{X/T},\QQ)^{W}.
	$$ 
	Now assume that $S=\Spec(k)$ is the spectrum of a field. Applying the latter equivalence to motivic cohomology yields 
	$$
		A^{n}_{N}(X,m)\cong A^{n}_{T}(X,m)^{W}
	$$
	for the equivariant intersection theory of $X$.
\end{example}

\begin{rem}
\label{rem.invariants.Tate}
Let us remark that taking homotopy fixed points under a finite group is a finite limit. Since the $\infty$-category of Tate-motives is stable it is closed under finite limits. Thus, if $\Gamma$ is a finite group acting on a Tate motive $M$, then also $M^{\Gamma}$ is Tate.\par 
Now consider the situation of the example above. Then we see that if $R\Gamma_{S}(\quot{X/T}\QQ) $ is Tate, then so is $R\Gamma_{S}(\quot{X/N}\QQ) $ (similarly for $R\Gamma$).
\end{rem}

\subsection{The relation between the motives of $\quot{X/T}$ and $\quot{X/G}$}
\label{sec.N}
Our goal in this subsection is to show that the map $\quot{X/N}\rightarrow \quot{X/G}$ yields and equivalence of motives up to Weyl-invariants. This will be achieved by analyzing the motive of $G/N$. Using that $G/T\rightarrow G/N$ is again a $W$-torsor, we will reduce to the case of the flag variety $G/B$. For this, we will need that motivic cohomology does not see the action of split unipotent subgroups.\par
Now let us recall the definition of a split unipotent subgroup. These are extensions of vector bundles, e.g. a Borel $B$ containing a maximal split torus $T$ is an extension of $T$ by a split unipotent subgroup.

\begin{defi}
	An algebraic $S$-group scheme $U$ is called \textit{split-unipotent} if there exists a normal series, i.e. a filtration $U= U_{n}\supseteq U_{n-1}\supseteq\dots\supseteq U_{0} = 0$ such that $U_{i}$ is normal in $U_{i+1}$, with each successive quotient is isomorphic a vector bundle $\VV(\Ecal)$, where $\Ecal$ is a finite locally free $\Ocal_{S}$-module.
\end{defi}

\begin{example} The $S$-subgroup scheme of unipotent upper triangular matricies $\UU_{n,S}$ in $\GL_{n,S}$ is split unipotent. More generally, let $G$ be a reductive $S$-group scheme and $P$ be a parabolic in $G$, then the unipotent radical $R_{u}(P)$ of $P$ is split unipotent (cf. \cite[Exp. XXVI Prop. 2.1]{SGA3}).
\end{example}

\begin{lem}
\label{lem.unip.ext}
	Let $G$ be a smooth linear algebraic $S$-group scheme. Consider a split exact sequence of $S$-group schemes
	$$
	1\rightarrow U\rightarrow G\rightarrow H\rightarrow 1
	$$
	where $U$ is split unipotent. Choose a splitting $H\hookrightarrow G$. Let $X$ be an $S$-scheme locally of finite type with an $G$-action. Let us denote the induced map $\quot{X/H}\rightarrow \quot{X/G}$ by $\pi$. Then the $!$-pullback induces a fully faithful functor $$\pi^{!}\colon \DM(X/G)\hookrightarrow \DM(X/H).$$
\end{lem}
\begin{proof}
	We have to check that the unit $\id\rightarrow \pi_{!}\pi^{!}$ is an equivalence.
	The natural map $\pi$ is a $U$-bundle. In particular, by \'etale descent it is enough to show that the $!$-pullback along $U\rightarrow S$ is fully faithful. But this follows from homotopy invariance.
\end{proof}

Let $G$ be a split reductive $S$-group scheme and let $B$ be a Borel containing $T$ inside $G$. Let $X$ be an $S$-scheme locally of finite type with $B$-action. The above lemma shows that $\DM(\quot{X/B})\hookrightarrow \DM(\quot{X/T})$. We also have that $G/T\rightarrow G/N$ is a $W$-torsor. The results on torsors under finite groups combined with the above lemma then yields that $R\Gamma_{S}(G/N,\QQ) \simeq R\Gamma_{S}(G/B,\QQ)^{W}$. The next result will use this fact.

\begin{prop}
\label{prop.1.BG}
	Let $G$ be a reductive $S$-group scheme with maximal torus $T$. Let $N$ denote the normalizer of $T$ in $G$. Further, let $X$ be an $S$-scheme locally of finite type with $G$-action and $f\colon \quot{X/N}\rightarrow\quot{X/G}$ the canonical map. Then the unit $1_{\quot{X/G}}\rightarrow f_{*}f^{*}1_{\quot{X/G}}$ is an equivalence. In particular, $f_{*}1_{\quot{X/N}}$ is a Tate motive in $\DM(\quot{X/G})$.
\end{prop}

Let us summarize the idea of the proof. We will show that the unit $1_{\quot{X/G}}\rightarrow f_{*}f^{*}1_{\quot{X/G}}$ is an equivalence. To see this we will use that $\quot{X/N}\rightarrow \quot{X/G}$ is a $G/N$-bundle.
 In particular, after \'etale-descent, we may assume that $f$ is given by the projection $G/N\rightarrow S$. Again, by \'etale descent, we may assume that $G$ is split with respect to $T$. Let $B$ be a Borel of $G$ containing $T$. Using our calculations about torsors under finite groups, it is enough to show that the induced map $1_{S}\rightarrow R\Gamma_{S}(G/T,\QQ)^{W}$ is an equivalence. But as $R\Gamma_{S}(G/T,\QQ)^{W}\simeq R\Gamma_{S}(G/B,\QQ)^{W}$ we will reduce this question to a classical question about Chow rings, at least when $S$ is the spectrum of a field, namely if the pullback map
$$
	\textup{CH}_{n}(S,m)_{\QQ}\rightarrow \textup{CH}_{n}(G/B,m)_{\QQ}^{W}
$$
is an isomorphism for all $n,m\in \ZZ$. But as we can identify $\textup{CH}_{n}(G/B,m)_{\QQ}^{W}$ with $\textup{CH}_{n}(G/G,m)_{\QQ}$, this is clear (cf. \cite[Cor. 8.7]{KriC}). In the case when $S$ is not the spectrum of a field, we have to use $K$-theory and as the flag variety admits an integral model, we can reduce to $S=\Spec(\ZZ)$. Up to $K_{1}$ the rational $K$-theory of $\ZZ$ and $\QQ$ are isomorphic. Using that the flag variety admits a stratification by affine spaces, we may assume the $S=\Spec(\QQ)$, as the only part we have to worry about, $K_{1}(\ZZ)_{\QQ}$, vanishes.

\begin{proof}[Proof of Proposition \ref{prop.1.BG}]
Throughout this proof, we will denote for readability the structure map of an Artin stack $\Xcal$ to $S$ with $p_{\Xcal}$.\par
	To show that the unit  $$1_{\quot{X/G}} \rightarrow f_{*}f^{*}1_{\quot{X/G}}$$ is an equivalence, we may  assume by \'etale descent that the map $f$ is given by the structure map $G/N\rightarrow S$. Again, by \'etale descent, we may assume that $G$ is split with respect to $T$. Then it is enough to show that the unit $1_{S}\rightarrow R\Gamma_{S}(G/N,\QQ)$ is an equivalence, since the pullback of this equivalence along $p_{X}$ yields the desired equivalence. \par 
	Note that the natural map $g\colon G/T\rightarrow G/N$ is naturally a $W$-torsor. Hence, by Lemma \ref{lem.torsor.fixed}, we see that $1_{G/N}\simeq (g_{*}1_{G/T})^{W}$ and thus $R\Gamma_{S}(G/N,\QQ)\simeq p_{G/N*}(g_{*}1_{G/T})^{W}$. Since $p_{G/N*}$ is a right adjoint it commutes with limits. Thus, we have 
	$$
		p_{G/N*}(g_{*}1_{G/T})^{W}\simeq (p_{G/N*}g_{*}1_{G/T})^{W}\simeq (p_{G/T*}1_{G/T})^{W}\simeq R\Gamma_{S}(G/T,\QQ)^{W}
	$$
	By Lemma \ref{lem.unip.ext}, we see that $R\Gamma_{S}(G/T,\QQ)\simeq R\Gamma_{S}(G/B,\QQ)$. By Example \ref{ex.flag}, we have that 
	$$R\Gamma_{S}(G/B,\QQ) \simeq \bigoplus_{w\in W} 1_{S}\langle l(w)-n\rangle,
	$$
	where $n$ is the relative dimension of the flag variety $G/B$.  In particular, $R\Gamma_{S}(G/B,\QQ)$ and thus $R\Gamma_{S}(G/T,\QQ)$ is Tate. As the $W$-invariants are defined as an image of a map, we see that $R\Gamma_{S}(G/T,\QQ)^{W}$ is also Tate. The $\infty$-category of Tate motives over $S$ is the stable subcategory of $\DM(S)$ generated by $1(r)$, for $r\in \ZZ$. Therefore, the natural map $1_{S}\rightarrow R\Gamma_{S}(G/T,\QQ)^{W}$ is an equivalence if and only if the induced map 
	\begin{equation}
	\label{eq.BG.1}
		\Homline_{\DM(S)}(1_{S}(r),1_{S})\rightarrow \Homline_{\DM(S)}(1_{S}(r),R\Gamma_{S}(G/T,\QQ)^{W})
	\end{equation}	
is an equivalence for all $r\in \ZZ$ (here $\Homline$ denotes the inner-$\Hom$). If $r>0$, then the $m$-th homotopy group of the left hand side is isomorphic to $K_{-2r+m}(S)^{(-r)}$ and thus vanishes as the negative Adams eigenspaces vanish per definition (cf. \cite{Soule}). By Remark \ref{rem.fix} and the computation of $p_{G/B*} 1_{G/B}$, we see that 
	$$
		\pi_{m}\Homline_{\DM(S)}(1_{S}(r),R\Gamma_{S}(G/T,\QQ)^{W}) \cong (\bigoplus_{w\in W}K_{-2r+m}(S)^{(l(w)-n-r)})^{W},
	$$
	which also vanishes for $r> 0$ as $l(w)-n\leq 0$. Therefore, we may assume from now on that $r\leq 0$.
	Before continuing with the proof, let us first assume that $S=\Spec(k)$ is the spectrum of a field $k$. Then by the above, we are reduced to show that the restriction map
$$
A^{r}(S,m-2r)\rightarrow A^{r}(G/B,m-2r)^{W}
$$ 
is an isomorphism (by homotopy invariance and identification of motivic cohomology with higher Chow-theory). But as $A^{r}(G/B,m-2r)^{W}\cong A^{r}(G/G,m-2r) = A^{r}(S,m-2r)$ (cf. \cite[Cor. 8.7]{KriC}), the restriction above yields the identity on $A^{r}(S,m-2r)$ which is trivially an isomorphism. Hence, we are done in the case of $S=\Spec(k)$.\par
Next, let us show how to reduce to the case of $S=\Spec(k)$.
	As $r\leq 0$, we may use commutativity of $\Homline$ with limits, to see that (\ref{eq.BG.1}) is an equivalence if and only if 
	\begin{equation}
		\label{eq.BG.2}
		\Homline_{\DM(S)}(1_{S}\langle r\rangle,1_{S})\rightarrow \Homline_{\DM(S)}(1_{S}\langle r\rangle,R\Gamma_{S}(G/T,\QQ)^{W})
	\end{equation}
	is an equivalence.
		Further, since $S$ is finite dimensional (cf. \cite[Prop. 14.109]{WED}), it is a fact that for every $n\in \ZZ$, the groups $K_{n}(S)^{(i)}$ vanish for all but finitely many $i\in\ZZ$ (cf. \cite[\S 2]{Soule}). Therefore, the morphism (\ref{eq.BG.2}) is an equivalence for all $r\leq0$ if and only if 
	$$
		\bigoplus_{r\in \ZZ}\Homline_{\DM(S)}(1_{S}\langle r\rangle,1_{S})\rightarrow \bigoplus_{r\in \ZZ}\Homline_{\DM(S)}(1_{S}\langle r\rangle,R\Gamma_{S}(G/T,\QQ)^{W})
	$$
	is an isomorphism. Equivalently, we can write this morphism as
	$$
		\bigoplus_{r\in \ZZ}\Homline_{\DM(S)}(1_{S},1_{S}\langle r\rangle)\rightarrow \bigoplus_{r\in \ZZ}\Homline_{\DM(S)}(M_{S}(G/T),1_{S}\langle r\rangle)^{W}
	$$
	The motive $M_{S}(G/T)$ is a direct sum of shifts and twists of the unit of $S$ and therefore compact (cf. \cite[Thm. 11.1.13]{CD1}).
	By compactness of $1_{S}$ and $M_{S}(G/T)$, the above morphism is an equivalence if and only if the morphism
	$$
		\Homline_{\DM(S)}(1_{S},\bigoplus_{r\in \ZZ}1_{S}\langle r\rangle)\rightarrow \Homline_{\DM(S)}(M_{S}(G/T),\bigoplus_{r\in \ZZ}1_{S}\langle r\rangle)^{W}
	$$
	is an equivalence. By construction of the rational K-theory spectrum in $\DM(S)$, denoted by $\KGL_{S,\QQ}$, we see that $\bigoplus_{r\in \ZZ} 1_{S}\langle r\rangle\simeq \KGL_{S,\QQ}$ (cf. \cite[Lem. 14.1.4]{CD1}). Thus, as $G/T$ is representable by a scheme (cf. \cite[Exp. IX Thm. 5.1]{SGA3}), the $n$-homotopy group of the right hand side is isomorphic to $K_{n}(G/T)_{\QQ}^{W}$, the rational higher $K$-theory of $G/T$. Further, the properties of the $K$-theory spectrum yield that the induced morphism $K_{n}(S)_{\QQ}\rightarrow K_{n}(G/T)_{\QQ}$ is given by the pullback along $G/T\rightarrow S$ (cf. \cite[\S 13.1]{CD1}). Taking $W$-invariants yields the map
	\begin{equation}
	\label{eq.final.K}
		K_{n}(S)_{\QQ}\rightarrow K_{n}(G/T)_{\QQ}^{W}
	\end{equation}
	on the $n$-th homotopy groups.
	We have to show that this map is an isomorphism for all $n\geq 0$. In the following we will freely identify $K_{n}(G/T)_{\QQ}$ with $K_{n}(G/B)_{\QQ}$ and $K$-theory with $G$-theory as we are only working with smooth schemes. As $G$ is split reductive, there exists a model $(G_{\ZZ},B_{\ZZ},T_{\ZZ})$ over $\Spec(\ZZ)$ such that $G/B$ is the pullback of the associated flag variety $G_{\ZZ}/B_{\ZZ}$. 
	As $G/B$ (resp.  $G_{\ZZ}/B_{\ZZ}$) is stratified by affine spaces, we have an equivalence of $K$-theory spectra $K(G/B)_{\QQ}\simeq K(G_{\ZZ}/B_{\ZZ})_{\QQ}\otimes_{K(\ZZ)_{\QQ}}K(S)_{\QQ}$ (cf. \cite[Cor. 4.3]{Joshua}), here $K(-)_{\QQ}$ denotes the underlying $E_{\infty}$-ring spectrum in $\DM(\ZZ)$, in particular $K(\ZZ)_{\QQ} = \KGL_{\ZZ,\QQ}$ and $K(S)_{\QQ}$ (resp. $K(G/B)_{\QQ}$) is just the $*$-pushforward of $\KGL_{S,\QQ}$ (resp. $\KGL_{G/B,\QQ}$) to $\DM(\ZZ)$. We will show that the induced morphism
	$$
		K(S)_{\QQ}\rightarrow K(G/B)_{\QQ}^{W}
	$$
	is an equivalence if and only if 
	$$
		K(\ZZ)_{\QQ}\rightarrow K(G_{\ZZ}/B_{\ZZ})_{\QQ}^{W}. 
	$$
	 This follows from the following equivalence, 
	$$
		(K(G_{\ZZ}/B_{\ZZ})_{\QQ}\otimes_{K(\ZZ)_{\QQ}}K(S)_{\QQ})^{W}\simeq K(G_{\ZZ}/B_{\ZZ})_{\QQ}^{W}\otimes_{K(\ZZ)_{\QQ}}K(S)_{\QQ},
	$$
	which we will show holds true.
 This is shown similarly to the analogous facts on representations of finite groups, but we will give an argument.\par
	For a scheme $X$ let $\DK(X)$ of $X$ denote the $\infty$-category of $KGL_{X,\QQ}$-modules in $\DM(X)$. This can also be extended to stacks by \'etale descent.\footnote{Note that the \'etale localized $K$-theory spectrum of stacks does not agree with the genuine rational $K$-theory spectrum. As we will only need the existence of a six functor formalism, we will just ignore the technicalities.} By abuse of notation will denote the tensor unit of $DK(X)$ with $1_{X}$ and just for the following will even drop the $X$ in the notation, as it is clear in which categories the units are. Also, we will again denote the inner-$\Hom$  by $\Homline$. Let us fix the following diagram with pullback squares
	$$
	\begin{tikzcd}
		G_{\ZZ}/B_{\ZZ}\arrow[r,""]\arrow[d,""]&\ZZ\arrow[d,"g"]&\arrow[l,"",swap]S\arrow[d,""]\\
		\quot{W\bs G_{\ZZ}/B_{\ZZ}}\arrow[r,"b"]&\quot{\ZZ/W}\arrow[d,"f"]& \quot{S/W}\arrow[l,"a'",swap]\arrow[d,""]\\
		&\ZZ&\arrow[l,"a"]S,
	\end{tikzcd}
	$$
	here $W$ acts on $\ZZ$ resp. $S$ trivially and $a, b$ are the natural projection. Note that by Lemma \ref{lem.torsor.fixed}, we have that ${g_{*}g^{*}}^{W}\simeq \id$. Using this, smooth base change, projection formula\footnote{Note that $f$ is proper, as $W$ is a finite constant group scheme.} and adjunctions, we get the following chain of equivalences
	\begin{align*}
		(K(G_{\ZZ}/B_{\ZZ})_{\QQ}\otimes_{K(\ZZ)_{\QQ}}K(S)_{\QQ})^{W}&\simeq\Homline_{\DK(\quot{W/\ZZ})}(1,g_{*}g^{*}(b_{*}1\otimes a'_{*}1))^{W}\\
		&\simeq\Homline_{\DK(\quot{W/\ZZ})}(1,b_{*}1\otimes a'_{*}1)\\
		&\simeq \Homline_{\DK(\ZZ)}(1,f_{*}(b_{*}1\otimes a'_{*}1))\\
		&\simeq K(G_{\ZZ}/B_{\ZZ})_{\QQ}^{W}\otimes_{K(\ZZ)_{\QQ}}K(S)_{\QQ}.
	\end{align*}
	Note that we also have to use that $K(G_{\ZZ}/B_{\ZZ})_{\QQ}$ is a direct sum of the unit, to see that the $\Homline$ agrees with the tensor product of spectra (see below for an explicit description).\par 
	Thus, as the structure map $G/B\rightarrow S$ is induced via pullback, we see that $K(S)_{\QQ}\rightarrow K(G/B)^{W}_{\QQ}$ is induced via 
	$$
		K(\ZZ)_{\QQ}\otimes_{K(\ZZ)_{\QQ}}K(S)_{\QQ}\xrightarrow{\alpha^{W}\otimes \id_{S}} K(G_{\ZZ}/B_{\ZZ})^{W}_{\QQ}\otimes_{K(\ZZ)_{\QQ}} K(S)_{\QQ},
	$$
	where $\alpha\colon G_{\ZZ}/B_{\ZZ}\rightarrow \Spec(\ZZ)$ denotes the structure map. In particular, we may assume that $S=\Spec(\ZZ)$.\par 
	As $G/B$ admits an affine cell decomposition, we see that 
	$$
		K(G_{\ZZ}/B_{\ZZ})_{\QQ}\simeq \bigoplus_{w\in W} K(\ZZ)_{\QQ},\textup{ and } K(G_{\QQ}/B_{\QQ})_{\QQ}\simeq \bigoplus_{w\in W} K(\QQ)_{\QQ},
	$$
	 where $G_{\QQ}/B_{\QQ}$ denotes the generic fiber of $G/B$ (cf. \cite{YayT}). It is well known that for any $n\neq 1$ there is an isomorphism $K_{n}(\ZZ)_{\QQ} \cong K_{n}(\QQ)_{\QQ}$ and $K_{1}(\ZZ)_{\QQ} =0$. Thus, it is clear that $K_{1}(\ZZ)_{\QQ}\rightarrow K_{1}(G/B)_{\QQ}^{W}$ is an isomorphism. For the other $K$-groups, we can use the identification of $K_{n}(\ZZ)_{\QQ}$ with $K_{n}(\QQ)_{\QQ}$ and may assume that $S=\Spec(\QQ)$. This concludes the proof via the argument in the Chow group case.
\end{proof}

Combining Proposition \ref{prop.1.BG} and the results of Section \ref{sec.fin} now yields the main result of this section. 

\begin{thm}
\label{thm.motive.BG}
Assume that $S$ is connected. Let $G$ be a reductive $S$-group scheme that admits a split maximal torus $T$ with Weyl group $W$. Assume $G$ acts on an $S$-scheme $X$ locally of finite type. Then we have
	$$
		R\Gamma_{S}(\quot{X/G},\QQ)\simeq R\Gamma_{S}(\quot{X/T},\QQ)^{W}.
	$$
\end{thm}
\begin{proof}
Let $N$ be the normalizer of $T$ in $G$. We can factor the map in the theorem as $\quot{X/T}\xrightarrow{f} \quot{X/N}\xrightarrow{g} \quot{X/G}$.  We now apply Lemma \ref{lem.torsor.fixed} and Proposition \ref{prop.1.BG} and get
	 \begin{align*}
	 	(p_{\quot{X/T}*}1_{\quot{X/T}})^{W}\simeq (p_{\quot{X/G}*}g_{*}f_{*}1_{\quot{X/T}})^{W}&\simeq  p_{\quot{X/G}*}g_{*}(f_{*}1_{\quot{X/T}})^{W}\\ &\simeq p_{\quot{X/G}}g_{*}1_{\quot{X/N}} \simeq p_{\quot{X/G}*}1_{\quot{X/G}}.
	 \end{align*}
\end{proof}

Let $S=\Spec(k)$ be the spectrum of a field and $G$ a split reductive group over $k$ with split maximal torus $T$ and Weyl group $W$. Edidin and Graham have shown in \cite{EG} that for any smooth scheme $X$ with $G$-action, we have $A^{\bullet}_{G}(X)\cong A^{\bullet}_{T}(X)^{W}$. Later on Krishna has shown in \cite[Cor. 8.7]{KriC} that this result can be generalized to \textit{higher} Chow groups, at least when $X$ is quasi-projective. Our Theorem \ref{thm.motive.BG} is the motivic analog of both of these results. In fact, we will see in the next corollary, that applying Theorem \ref{thm.motive.BG} to motivic cohomology recovers both Edidin-Graham's and Krishna's result. Moreover, as any reductive group over a field splits after passage to a finite Galois extension, we can use the results of Section \ref{sec.fin} to extend this comparison of (higher) equivariant Chow groups to \textit{arbitrary} reductive groups over fields. We also note that Levine showed that Bloch's higher Chow groups also satisfy a localization sequence, even in the non quasi-projective case. The following corollary summarizes this discussion.

\begin{cor}
\label{cor.chow.bg}
Assume that $S=\Spec(k)$ is the spectrum of a field.
	Let $G$ be a reductive $S$-group scheme with maximal torus $T$. Let $W$ denote the absolute Weyl\footnote{Let $\kbar$ be an algebraic closure of $k$, then $\Wcal\coloneqq N_{G}(T)/T$ is a finite \'etale group scheme of $S$ and we set the absolute Weyl group to be $W\coloneqq \Wcal(\kbar)$} group of $T$ in $G$. Assume $G$ acts on a smooth finite type $S$-scheme $X$. Then we have 
	$$R\Gamma_{S}(\quot{X/G},\QQ)\simeq R\Gamma_{S}(\quot{X/T},\QQ)^{W}\coloneqq \colim_{K/k\textup{ finite Galois}} (R\Gamma_{S}(\quot{X_{K}/T_{K}},\QQ)^{W_{K}})^{\Gal(K/k)}.$$
	In particular, applying this result to motivic cohomology yields for all $n,m\in \ZZ$
	$$
		A^{n}_{G}(X,m)_{\QQ}\cong A^{n}_{T}(X,m)_{\QQ}^{W}.
	$$
\end{cor}
\begin{proof}
Any reductive group over $k$ becomes split after passing to a finite Galois extension $K/k$ (cf. \cite[Exp. XXII Cor. 2.4]{SGA3}). Thus, let $K$ be such an extension, so that $T_{K}$ is a split maximal torus. Then we have the following diagram with pullback squares
$$
	\begin{tikzcd}
		\quot{X_{K}/T_{K}}\arrow[r,""]\arrow[d,"f"]\arrow[rr,"p_{T_{K}}", bend left = 2em]& \quot{X_{K}/G_{K}}\arrow[d,"g"]\arrow[r,"p_{G_{K}}"] &\Spec(K)\arrow[d,"p_{K}"]\\
		\quot{X/T}\arrow[r,""]\arrow[rr,"p_{T}", bend right = 2em]&\quot{X/G}\arrow[r,"p_{G}"]& \Spec(k).
	\end{tikzcd}
$$
As $K/k$ is a finite Galois extension, the morphisms $p_{K}$ and thus also $f$ are $\Gamma\coloneqq \Gal(K/k)$-torsors. Then Lemma \ref{lem.torsor.fixed} yields the equivalence 
$$p_{G*} 1_{\quot{X/G}} \simeq (p_{K*}p_{K}^{*}p_{G*} 1_{\quot{X/G}})^{\Gamma}.$$
As the diagram above has cartesian squares, we can use smooth base change, to see that $p_{K}^{*}p_{G*}\simeq p_{G_{K}*}g^{*}$. But by Theorem \ref{thm.motive.BG}, we have
$$
p_{G_{K}*}g^{*}1_{\quot{X/G}}\simeq R\Gamma_{K}(\quot{X_{K}/G_{K}},\QQ)\simeq R\Gamma_{K}(\quot{X_{K}/T_{K}},\QQ)^{W_{K}}.
$$
Commutativity of the above diagram yields $p_{K*}R\Gamma_{K}(\quot{X_{K}/T_{K}},\QQ)\simeq p_{T*}f_{*}1_{\quot{X_{K}/T_{K}}}.$ Thus, we have 
$$
	p_{G*}1_{\quot{X/G}}\simeq ((p_{T*}f_{*}1_{\quot{X_{K}/T_{K}}})^{W_{K}})^{\Gamma}.
$$ This equivalence is compatible with passing to finite Galois extensions $K'/k$ containing $K$. This yields the first equivalence.
\par
The result about motivic cohomology follows from Remark \ref{rem.fix} by using \cite[Thm. 2.2.10]{RS1} which proves that motivic cohomology for smooth quotient stacks is computed by the higher equivariant Chow groups of Edidin-Graham.\footnote{In \textit{loc.cit.} they assume some properties on the groups and on $X$, as Edidin-Graham need these assumptions to compare higher Chow theory of stacks and equivariant higher Chow theory \cite[Prop. 13 (b)]{EG}. The assumptions in \textit{loc.cit.} are needed as Bloch only shows the existence of a long exact sequence for higher Chow groups in the case where $X$ is \textit{quasi-projective}. This result was extended by Levine to all \textit{finite type} $k$-schemes (cf. \cite{Levine2}). Thus, the comparison of Edidin-Graham and hence also of Richarz-Scholbach go through in the case of the corollary.}
\end{proof}

\section{Torsors under tori}
\label{sec.T-tors}
Let us fix a split torus $T$ over $S$ (cf. \cite[Exp. IX]{SGA3}).\par
In this subsection, we want to understand the motivic homotopy theory of torsors under $T$. More precisely, let us consider the following situation. Let $X\rightarrow Y$ be a $T$-torsors of Artin $S$-stacks locally of finite type. We want to understand the relation between $M_{S}(X)$ and $M_{S}(Y)$.\par 
Let us recall the classical case of Chow theory. For this paragraph let us assume that $X$ and $Y$ are smooth Artin stacks over $\Spec(k)$, where $k$ is a field. Then the group $\That$ of characters of $T$ acts on $A^{\bullet}(Y)$ in the following way. Let $\chi\in\That$ be a character and consider its associated $1$-dimensional representation $\kappa(\chi)$. The quotient $L_{\chi}\coloneqq X\times^{T}\kappa(\chi)$ is representable by a line bundle over $Y$. Multiplication with the first Chern class of $L_{\chi}$ yields an action of $\That$ on $A^{\bullet}(Y)$. It is known that in this case, we have
$$
	A^{\bullet}(X)\cong A^{\bullet}(Y)/\That A^{\bullet}(Y).
$$
\par
Our goal is to extend this result to motivic homotopy theory, so that it generalizes the result about Chow theory and yields a similar statement for $K$-theory. Even though we work with Beilinson motives, we will remark in the end how to extend this to integral $K$-theory under some assumptions on $X$ and $Y$.

\subsection{The motive of $T$-torsors}

Let us denote the character group of $T$ with $\That$. Let $X\rightarrow Y$ be a $T$-torsor of Artin stacks locally of finite type over $S$ and $\chi\in\That$ a character. Let $\GG_{\textup{m},S}$ act via left multiplication in $\AA^{1}_{S}$. Then $T$ acts via $\chi$ on $\AA^{1}_{S}$ and thus, $X\times^{T}\AA^{1}_{S}\rightarrow X/T\cong Y$ yields a line bundle over $Y$. The action of the first Chern class of $X\times^{T}\AA^{1}_{S}$ on the motivic cohomology will be described by a Gysin sequence (cf. Proposition \ref{prop.T-tors}).

\begin{notation}
\label{not.char}
	In the following we want to split up the $T$-torsor $f\colon X\rightarrow Y$ into a sequence of $\GG_{\textup{m},S}$-torsors. Note that by splitness $T\cong \GG_{\textup{m},S}^{r}$ for some $r\in \NN$. Fixing a numbering of the $\GG_{\textup{m},S}$-components of $T$, we can embed for any $1\leq k\leq r$ the product $\GG_{\textup{m},S}^{k}$ into $T$ by $\id_{\GG_{\textup{m},S}^{k}}\times 1^{r-k}$. Then $\GG_{\textup{m},S}^{k}$ acts on $X$ via this embedding. We get a sequence 
	$$
		X\rightarrow X/\GG_{\textup{m},S}\rightarrow X/\GG_{\textup{m},S}^{2}\rightarrow \dots \rightarrow X/T\cong Y
	$$
	of $\GG_{\textup{m},S}$-torsors. We denote the induced maps $X/\GG_{\textup{m},S}^{i}\rightarrow Y$ with $f_{i}$.
\end{notation}

To ease notation, we will denote the $!$-push/pull of the unit along a map of Artin stacks $f\colon X\rightarrow Y$ by $M_{Y}(X)\coloneqq f_{!}f^{!}1_{Y}$.

\begin{prop}
\label{prop.T-tors}
	Assume that $T\cong \GG_{\textup{m},S}^{r}$ for some $r\in \NN$. Let $X\rightarrow Y$ be a $T$-torsor of smooth Artin stacks over $S$. Then, there exists a filtration
	$$
		M_{Y}(X)=M_{0}\rightarrow M_{1}\rightarrow\dots\rightarrow M_{r} = 1_{Y}
	$$ 
	in $\DM(S)$, where $M_{i}\coloneqq M_{Y}(X/\GG_{\textup{m},S}^{i})$ such that the cofiber of $M_{Y}(X_{i-1})\rightarrow  M_{Y}(X_{i})$ is given by $M_{Y}(X_{i})\langle 1\rangle$ and the map $ M_{Y}(X_{i})\rightarrow  M_{Y}(X_{i})\langle 1\rangle$ is induced by multiplication with $c_{1}(X_{i-1}\times^{\GG_{\textup{m}}}\AA^{1}_{S})$.
\end{prop}
\begin{proof}
	This follows by successively using \cite[Prop. 2.32]{HosLeh}. But for completion let us give a proof by recalling the argument.\par 
	 The morphism $X_{i-1}\rightarrow X_{i}$ is a $\GG_{m}$-torsor. In particular, the scheme $\Lcal_{i}\coloneqq X_{i-1}\times_{S}^{\GG_{\textup{m},S}} \AA^{1}_{S}$ is a line bundle over $X_{i}$. Let $s\colon X_{i}\hookrightarrow V_{i}$ denote the zero section. Certainly, the complement of the closed immersion $s$ is isomorphic to $X_{i-1}$.
	  Then the Gysin sequence of Lemma \ref{Gysin for stacks} yields a fibre sequence
	 $$
	 	M_{Y}(X_{i-1})\rightarrow M_{Y}(V_{i})\simeq M_{Y}(X_{i})\xrightarrow{\varphi} M_{Y}(X_{i})\langle 1\rangle .
	 $$
	By construction of the Gysin sequence $\varphi$ is given by multiplication with $c_{1}(\Lcal_{i})$. This concludes the proof.
\end{proof}

\begin{cor}
\label{cor.T-tors.tate}
	Let $f\colon X\rightarrow Y$ be a $T$-torsor of smooth Artin stacks over $S$. Then $f$ is a Tate map.
\end{cor}
\begin{proof}
	 Proposition \ref{prop.T-tors} implies that $M_{Y}(X)$ is a successive extension of $1_{Y}$. Thus, the result follows from Remark \ref{rem.tate}.
\end{proof}

\subsection{Motivic cohomology of $T$-torsors}
 Let us fix an $r\in \NN$ such that $T\cong \GG_{\textup{m},S}^{r}$.\par
For any Artin stack $X$ locally of finite type over $S$, we will denote its motivic cohomology with
$$
	H^{p}(X,\QQ(n))\coloneqq \Hom_{\DM(S)}(M_{S}(X),\QQ(n)[p]).
$$
If $X$ is representable by a smooth $S$-scheme, then we have $H^{p}(X,\QQ(n))\cong K_{2n-p}(X)^{(n)}$.
If $S=\Spec(k)$ is the spectrum of a field and $X$ is smooth and over $S$, we have $H^{p}(X,\QQ(n)) \cong A^{n}(X,2n-p)$. Note that for a smooth Artin $S$-stack $X$ the motivic cohomology vanishes automatically in certain degrees by descent and vanishing of negative $K$-theory for regular schemes, i.e. we have that $H^{p}(X,\QQ(n))\cong 0$ for $p>2n$. 

\begin{notation}
	Consider a $\chi\in \That$. The associated line bundle $L_{\chi}\coloneqq X\times_{S}^{T} \VV(\Ecal_{\chi})$ yields a map $H^{p+2}(Y,\QQ(n+1))\rightarrow  H^{p}(Y,\QQ(n))$, by multiplication with the Chern class of $L_{\chi}$. We denote the image of this map with $c_{1}(L_{\chi})H^{p}(Y,\QQ(n))$. 
\end{notation}

\begin{prop}
\label{prop.T-tors.coho}
	In the setting of Proposition \ref{prop.T-tors} let us further fix an $n\in \ZZ$. Then, we have
	$$
		H^{2n}(X,\QQ(n)) \cong H^{2n}(Y,\QQ(n))/\That H^{2n}(Y,\QQ(n)).
	$$
\end{prop}
\begin{proof}
First let us note that it is enough\footnote{Any character is generated by primitive characters and the corresponding $1$-dimensional representation is given by the associated tensor product.} to show that 
	$$
		H^{2n}(X,\QQ(n)) = H^{2n}(Y,\QQ(n))/\langle c_{1}(X\times_{S}^{T}\VV(\Ecal_{\chi_{i}}))H^{2n}(Y,\QQ(n))\rangle ,
	$$
	where $\chi_{i}$ is a primitive character in $\That$.
	Let $X=X_{0}\rightarrow X_{1}\rightarrow\dots\rightarrow X_{r} = Y$ be the sequence of Proposition \ref{prop.T-tors}. For each $0\leq i\leq r$ this yields a long exact sequence on motivic cohomology
	\begin{align*}
		\dots\rightarrow  H^{2n+2}(X_{i},\QQ(n+1))&\rightarrow H^{2n}(X_{i},\QQ(n))\\ &\rightarrow H^{2n}(X_{i-1},\QQ(n))\rightarrow H^{2n+3}(X_{i},\QQ(n+1))\rightarrow \dots .
	\end{align*}
	We have $H^{2n+3}(X_{i},\QQ(n+1))= 0$ and thus get an exact sequence of the form
	$$
		H^{2n+2}(X_{i},\QQ(n+1))\xrightarrow{a} H^{2n}(X_{i},\QQ(n))\xrightarrow{b} H^{2n}(X_{i-1},\QQ(n))\rightarrow 0.
	$$ 
	The map $b$ is the usual pullback on motivic cohomology. The map $a$ is induced by multiplication with the Chern class of the line bundle $\Lcal_{i} = X_{i-1}\times^{\GG_{\textup{m}}}_{S}V(\Ecal_{\chi_{i}})$. As $X_{r} = Y$, we have $H^{2n}(X_{r-1},\QQ(n)) \cong H^{2n}(Y,\QQ(n))/c_{1}(\Lcal_{r}) H^{2n}(Y,\QQ(n))$. Hence, inductively we see that 
	$$
		H^{2n}(X,\QQ(n)) \cong H^{2n}(Y,\QQ(n))/\langle c_{1}(\Lcal_{i}) H^{2n}(Y,\QQ(n))\rangle_{1\leq i\leq r}.
	$$
	We are left to show that $c_{1}(\Lcal_{i}) H^{2n}(Y,\QQ(n)) = c_{1}(X\times_{S}^{T}V(\Ecal_{i}))H^{2n}(X,\QQ(n))$. For this let us start with $i= r$. Then by construction $X\times_{S}^{T} V(\Ecal_{\chi_{r}}) \cong X_{r-1}\times_{S}^{\GG_{\textup{m}}} V(\Ecal_{\chi_{r}})$. Inductively, we may replace $Y$ by $X_{i}$, where the claim again follows by construction.\par
\end{proof}

\begin{cor}
\label{cor.T-tors.coho}
	Let $S=\Spec(k)$ be the spectrum of a field and Let $X\rightarrow Y$ be a $T$-torsor of smooth Artin stacks of over $S$. Then 
	$$
		A^{\bullet}(X)_{\QQ}\cong A^{\bullet}(Y)_{\QQ}/\That A^{\bullet}(Y)_{\QQ}.
	$$
\end{cor}
\begin{proof}
	This follows immediately from Proposition \ref{prop.T-tors.coho}.
\end{proof}

\begin{rem}
\label{rem.scalloped}
Proposition \ref{prop.T-tors} and Proposition \ref{prop.T-tors.coho} can be extended to other cohomology theories in the following way. Let us fix a $T$-torsor $X\rightarrow Y$ of smooth Artin $S$-stacks.\par 
\begin{enumerate}
	\item[(1)] (Rational \'etale localized cohomology theories) 	Let us fix a $T$-torsor $X\rightarrow Y$ of smooth Artin $S$-stacks. Let $M\in \SH(S)_{\QQ,\et}$ be an oriented $E_{\infty}$-ring spectrum and let us denote its pullback to any smooth Artin $S$-stack $Z$ with $M_{Z}$. The orientation of $M$ yields a Chern class map 
	$$
		c_{1}\colon \Pic(Z)\rightarrow H^{2}(Z,M(1))\coloneqq \Hom_{\SH(Z)}(1_{Z},M_{Z}(1)[2]).
	$$
	In the same fashion as before, for any character $\chi\in\That$, we can define $c_{1}(L_{\chi})H^{n}(Y,M)$. \par
	Assume there exists a $p\in\ZZ$ such that $M$ satisfies $$H^{m}(Z,M)\coloneqq \Hom_{\SH(Z)}(1_{Z},M_{Z}[m])=0$$ for all $m>p$. Then 
	$$
		H^{p}(X,M) \cong H^{p}(Y,M)/\langle c_{1}(X\times_{S}^{T}\VV(\Ecal_{\chi_{i}}))H^{p}(Y,M)\rangle.
	$$
	If we take for example $M=\KGL_{\QQ,S}$, the rational $K$-theory spectrum in $\SH(S)$, then we get 
	$$
		K_{0}(X)_{\QQ}^{\et}\cong K_{0}(Y)_{\QQ}^{\et}/\langle c_{1}(X\times_{S}^{T}\VV(\Ecal_{\chi_{i}}))K_{0}(Y)_{\QQ}^{\et}\rangle,
	$$
	where $K_{0}(-)_{\QQ}^{\et}$ denotes the \'etale localized rational $K$-theory.
	\item[(2)] (Integral $K$-theory) The extension to integral $K$-theory is more subtle, as we have to restrict ourselves to certain algebraic stacks\footnote{These are called \textit{scalloped stacks} (cf. \cite{KR1}).} to make sense of the stable homotopy category. For simplicity, we may assume that $X=\quot{X'/H}$ and $Y=\quot{Y'/F}$ are represented by quotients of quasi-projective schemes by diagonalizable group schemes. Then there is a well defined notion of a stable homotopy category $\SH(X)$ resp. $\SH(Y)$ together with a functorial $E_{\infty}$-ring object $\KGL$ that represents equivariant $K$-theory (cf. \cite{KR1}). Further, Bott-periodicity yields an orientation on $\KGL$. In particular, using that $\KH(X)$ and $\KH(Y)$ are connective\footnote{This holds Nisnevich locally by \textit{loc.cit.} and by descent for $\KH$.} (cf. \cite[Thm. 5.7]{HoyKri}), we see that Corollary \ref{cor.T-tors.coho} holds for integral $K_{0}$ in this case, i.e. 
	$$
		K_{0}(X)\cong K_{0}(Y)/\That K_{0}(Y).
	$$
	This result can be glued to the class of so called \textit{scalloped stacks} (cf. \cite{KR1} for the notion of scalloped stacks and the construction of $\SH$). 
\end{enumerate}

\end{rem}

\section{Motivic cohomology of quotients up to isogeny}
\label{sec.mot.isog}
We are now ready to answer the main question of this article (\ref{main-question}). Let us assume that $S$ is connected and let us fix a reductive $S$-group scheme $G$ and assume that it admits a maximal split torus $T$. The non-split case will be discussed in Remark \ref{rem.non.split}.\par 
For completeness, we will work in a more general setting than in the beginning of the introduction.\par
Let $P$ resp. $Q$ be parabolics inside $G$ with Levi-components $L$ resp. $M$. Assume that $T$ is contained in $L$. Let $\varphi\colon L\rightarrow M$ be an isogeny. Then $L$ acts on $G$ via $(l,g)\mapsto lg\varphi(l)^{-1}$. We are interested in the quotient of this action, which we denote by $\quot{G/_{\varphi}L}$ or rather its motive. To do so, we follow the idea of Brokemper in the proof of \cite[Prop. 1.2]{Bro2}.\par As $\varphi$ is an isogeny the image of $T$ is again a maximal torus. In particular, up to conjugation by an element $g_{0}\in G(S)$, we may identify $T$ with $\varphi(T)$. The $g_{0}$-conjugation of $G$ induces an isomorphism $G\rightarrow G$ that is $L$-equivariant, where $L$ acts on the right hand side via $(l,g)\mapsto lgg_{0}^{-1}\varphi(l)^{-1}g_{0}$. In particular, after replacing $M$ resp. $Q$ by their $g_{0}^{-1}$-conjugation, we may assume that $\varphi(T) = T$. Then we have the following embedding $T\hookrightarrow T\times T$, via $t\mapsto (t,\varphi(t))$. The quotient under this embedding is $T\times T/T\cong T$. Thus, the naturally induced morphism $\quot{G/_{\varphi}T}\rightarrow \quot{G/T\times T}$, where $T\times T$ acts on $G$ via $(t,t',g)\mapsto tgt'^{-1}$, is a $T$-torsor. This leaves us with the following picture
$$
	\begin{tikzcd}
		 \quot{G/_{\varphi}L}&\arrow[l,"a",swap] \quot{G/_{\varphi} T}\arrow[r,"b"]&\quot{G/T\times T}.
	\end{tikzcd}
$$
The morphism $b$ is by the above a $T$-torsor. Therefore, we can use Corollary \ref{cor.T-tors.tate} to see that $b$ is also Tate. \par
If $T$ is split, we can say more.  Morphism $a$ yields us an equivalence of motives of $\quot{G/_{\varphi}L}$ and $\quot{G/_{\varphi} T}$ up to Weyl invariance, as in the Chow group case. To be more precise, we have 
$$
	R\Gamma_{S}(\quot{G/_{\varphi}L},\QQ)\simeq R\Gamma_{S}(\quot{G/_{\varphi}T},\QQ)^{W_{L}},
$$
where $W_{L}$ denotes the Weyl group of $T$ inside $L$ by Corollary \ref{cor.chow.bg}. However if $S = \Spec(k)$, the spectrum of a field, we may drop the assumption that $T$ is split and work with the absolute Weyl group instead.\par
We can compute the motive resp. the motivic cohomology of $\quot{G/_{\varphi} T}$ via the motive resp. the motivic cohomology of $\quot{G/T\times T}$ using Proposition \ref{prop.T-tors} and Proposition \ref{prop.T-tors.coho}. \par If $G$ admits a Borel $B$ containing $T$, then by invariance under extensions of unipotent groups, we have a fully faithful functor $\DM(\quot{T\bs G/B})\hookrightarrow \DM(\quot{G/T\times T})$ (cf. Lemma \ref{lem.unip.ext}). Therefore, with all of the above we see that the $T$-equivariant motivic cohomology resp. motive of the flag variety $G/B$ yields results about the motivic cohomology resp. the motive of $\quot{G_{\varphi}/L}$. But the author has shown\footnote{Even though in the cited article, we assume that $S$ is affine, all the arguments go through in $\DM$ without this assumption.} in \cite{YayT} that the motive of $\quot{T\bs G/B}$ is computed by $M_{S}(G/B)\otimes M_{S}(BT)$. \par 
If further $S=\Spec(k)$, we have seen that
	$$
		M_{S}(\B{T})\cong \bigotimes_{i=1}^{r}M_{S}(\B{\GG_{\textup{m}}}) = \bigotimes_{i=1}^{r}\bigoplus_{j\geq 0}1_{k}\langle i\rangle,
	$$
which is completed Tate (cf. Example \ref{ex.BGL}). As the motive of the flag variety $G/B$ is also Tate (cf. 
Example \ref{ex.flag}), we see that $M_{S}(T\bs G/B)$ is completed Tate. Summarizing all of the above yields the following theorem.

\begin{thm}
\label{thm.main}
Assume that $S=\Spec(k)$ is the spectrum of a field. 
Then the motives $R\Gamma_{S}(\quot{G/_{\varphi}L},\QQ)$ and $M_{S}(\quot{G/_{\varphi}L})$ are completed Tate motives in $\DM(S)$. We can compute the Chow ring of $\quot{G/_{\varphi}L}$ as
	$$
		A^{\bullet}(\quot{G/_{\varphi}L})_{\QQ}\cong \left( A^{\bullet}_{T}(G/T)_{\QQ}/\That A^{\bullet}_{T}(G/T)_{\QQ}\right)^{W_{L}},
	$$
	where $W_{L}$ denotes the Weyl group of $T$ inside $L$.
\end{thm}
\begin{proof}
	This follows from the discussion above.
\end{proof}

\begin{rem}
	The isomorphism in Theorem \ref{thm.main} is also valid in the case, where $S$ is not the spectrum of a field, after replacing the Chow groups with the $(2n,n)$ motivic cohomology group for $n\in \ZZ$ as in Proposition \ref{prop.T-tors.coho}.
\end{rem}

\begin{rem}
\label{rem.Chow.Bro}
	If $G$ is a split reductive group over a field with split maximal torus $T$, Brokemper has shown that one can give a more explicit computation of the Chow ring of $\quot{G_{\varphi}/L}$ using the computations of Brion \cite{Brion} (cf. \cite[Prop. 1.2]{Bro2}). To be more precise, we can write the last isomorphism of Theorem \ref{thm.main} as
	$$
		A^{\bullet}(\quot{G/_{\varphi}L})_{\QQ}\cong S^{W_{L}}/(f-\varphi f\mid f\in S_{+}^{W_{G}}),
	$$
	where $S=\Sym_{\QQ}(\That) \cong A^{\bullet}_{T}(\ast)_{\QQ}$, $S_{+}$ are the elements of positive degree and $W_{G}$ is the Weyl group of $T$ in $G$. A more detailed computation can be found in the proof of \cite[Prop. 1.2]{Bro2}.
\end{rem}

Remark \ref{rem.scalloped} enables us to get an analogous result of Remark \ref{rem.Chow.Bro} for rational $K_{0}$.

\begin{prop}
\label{prop.chev}
	In the setting above, let us assume that $S=\Spec(k)$ is the spectrum of a field, the derived group of $G$ is simply connected and $T$ is a split maximal torus. Then we have
	$$
		K_{0}(\quot{G/_{\varphi}L})_{\QQ} \cong R(T)_{\QQ}^{W_{L}}/(f-\varphitilde f\mid f\in R(T)_{\QQ}^{W_{G}}).
	$$
\end{prop}
\begin{proof}
	For $K_{0}$, we could not produce an analogue of Corollary \ref{cor.chow.bg} using Theorem \ref{thm.motive.BG} and thus have to use a result of Krishna on equivariant $G$-theory (cf. \cite[Lem. 9.2]{KriRR}). For completion we recall the main argument.\par 
	First, we may replace $Q$ and $M$ by $\prescript{g_{0}}{}Q$ and $\prescript{g_{0}}{}M$ and assume that $\varphi(T) = T$. In particular, $\varphitilde$-conjugation of $G$ by $T$ is just $\varphi$-conjugation.
	\par  
	Now we embed $T$ into $T\times_{S}T$ by $t\mapsto (t,\varphi(t))$. Let $T\times_{S}T$ act on $G$ by $(t,t').g \coloneqq tgt'^{-1}$. This yields a morphism $\quot{G/_{\varphi}T}\rightarrow \quot{G/T\times_{S}T}$. This is $T\cong T\times_{S}T/T$-torsor. Thus, by Remark \ref{rem.scalloped}, we have 
	\begin{align*}
		K_{0}(\quot{G/_{\varphi} T})_{\QQ}&\cong K_{0}( \quot{G/T\times_{S}T})_{\QQ}/\That K_{0}( \quot{G/T\times_{S}T})_{\QQ}.
	\end{align*}
	By homotopy invariance, we have $K_{0}(\quot{G/T\times_{S}T})_{\QQ} \cong K_{0}^{T}(G/B)_{\QQ}$. Therefore, we are reduced to classical statements about $T$-equivariant $K$-theory of flag varieties (cf. \cite{Uma}) and get
	$$
		K_{0}(\quot{G/_{\varphi}T})_{\QQ}\cong R(T)_{\QQ}/(f-\varphitilde f\mid f\in R(T)_{\QQ}^{W_{G}}).
	$$
	It follows\footnote{Even though Krishna works with $G=\GL_{n}$ his proof of \cite[Lem. 9.2]{KriRR} goes through in our case.} from \cite[Lem. 9.2]{KriRR} that 
	$$
	K_{0}(\quot{G/_{\varphi}L})_{\QQ}\cong K_{0}(\quot{G/_{\varphi}T})_{\QQ}^{W_{L}}.
	$$
	Thus, it suffices to show that $IR(T)_{\QQ}^{W_{L}} = IR(T)_{\QQ}\cap R(T)^{W_{L}}_{\QQ}$, where $I=(f-\varphitilde f\mid f\in R(T)_{\QQ}^{W_{G}})$, but this follows from faithfully flatness of $R(T)^{W_{L}}\hookrightarrow R(T)$ (cf. \cite[Thm. 1.2]{Stein} - see also the proof of \cite[Prop. 1.2]{Bro2} resp. \cite[Prop. 2.3.2]{Bro1} for a detailed argument in the Chow group case).
\end{proof}

We want to give two motivating examples of quotients $\quot{G/_{\varphi}L}$ as above, that appear naturally and apply Theorem \ref{thm.main} to see that their motives are Tate. The classifying stack of finite groups of Lie-type and the stack of $G$-zips. Both are examples in characteristic $p>0$. In the last section, we want to use Theorem \ref{thm.main} to give an idea how we want to approach geometric representation theory of finite groups of Lie-type $G^{F}$ by relating it motivically to geometric representation of the Langlands dual of $G$ (see below for the notation).

\begin{example}
\label{ex.fin-group}
	Let us assume that $S=\Spec(\FF_{q})$ be a finite field of characteristic $p>0$. We set $\varphi\colon G\rightarrow G$ to be the $q$-Frobenius. This is an isogeny and thus, we can apply Theorem \ref{thm.main} in this setting. Further, the stack $\quot{G/_{\varphi}G}$  is isomorphic to $\B{G^{F}}$, where $G^{F}$ is the stabilizer group scheme of the neutral element (cf. \cite[Lem. 2.1]{Bro2}). It is well known that $G^{F}(\FFbar_{q})\cong G(\FF_{q})$, where $\FFbar_{q}$ denotes an algebraic closure of $\FF_{q}$. Thus, we see that the $\FF_{q}$-motive of the classifying stack of a finite group of Lie-type is completed Tate.
\end{example}

One of Brokemper's applications of his computations is the computation of the Chow ring of $G$-zips. In a similar fashion we will apply the above results and show that the motive of the stack of $G$-zips over a field is completed Tate.

\begin{example}
\label{ex.gzip.tate}
	Let $k$ be a field of characteristic $p>0$ and let $S=\Spec(k)$.
	Let $G, P,Q$ be as above. Let us denote the unipotent radical of $P$ resp. $Q$ with $R_{u}(P)$ resp. $R_{u}(Q)$. Further, let $\varphi\colon P/R_{u}(P)\rightarrow Q/R_{u}(Q)$ be an isogeny. The datum $\Zcal\coloneqq (G,P,Q,\varphi)$ is called \textit{algebraic zip-datum}. To every algebraic zip-datum like above, we can associate the group 
	$$
		E_{\Zcal}\coloneqq \{ (p,q)\in P\times Q\mid \varphi(\pbar) = \qbar\}.
	$$
	The group $E_{\Zcal}$ acts on $G$ via conjugation $(p,q).g \coloneqq pgq^{-1}$.  The quotient stack $\GZip\coloneqq\quot{G/E_{\Zcal}}$ is called the stack of \textit{G-zips}. There are also alternative constructions using a Tannakian formalism on the stack of $F$-zips (cf. \cite{PWZ}). In \textit{op.cit.} there is also an explicit description of the points of $\GZip$. Let $L\subseteq P$ be a Levi-component of $P$. Then as seen in the proof of \cite[Thm. 2.4.4]{Bro1} there is a split exact sequence
	$$
		1\rightarrow R_{u}(P)\times R_{u}(Q)\rightarrow E_{\Zcal}\rightarrow L\rightarrow 1,
	$$
	where the splitting is induced by $L\hookrightarrow E_{\Zcal}$, $l\mapsto (l,\varphi(l))$.
	Therefore, by homotopy invariance, we have $M_{S}(\GZip) \simeq M_{S}(\quot{G/_{\varphi} L})$ which is completed Tate by Theorem \ref{thm.main} and the discussion before the theorem.
\end{example}

\begin{rem}[The non-split case]
\label{rem.non.split}
Let us consider the setting above, where $T$ is not necessarily split. Then Corollary \ref{cor.T-tors.tate} does not work and we cannot prove that $M_{S}(\quot{G/_{\varphi}L}$ is Tate. This is because, we have to pass to a finite Galois covering of $S$ to obtain a splitting. However, the problem resolves once we enlarge the category of Tate motives by adding motives coming from finite Galois coverings. If $S=\Spec(k)$ these are known as \textit{Artin-Tate motives} and the relation with the motivic t-structure on Tate motives was explained by Wildeshaus \cite{WildArtin}\footnote{Wildehaus defines Artin motives as motives coming of $0$-dimensional projective schemes. The category of Artin-Tate motives is then generated by Tate and Artin motives.}. \par
Let $X$ be an $S$-Artin stack and $M\in \DM(X)$. Through Lemma \ref{lem.torsor.fixed}, we see that if there exists a finite Galois covering $S'\rightarrow S$ such that $M_{S'}$ is Tate, then $M$ is Artin-Tate (even though non-standard one can take this as a definition).  In particular, we see that $M_{S}(\quot{G/_{\varphi}L}$ is contained in the cocomplete stable subcategory generated by Artin-Tate motives, even if $T$ is not split.  
\end{rem}

\section{Generalizations}
\label{sec.gen}

In this section, we want to give an overview on the integral version of Theorem \ref{thm.motive.BG} and Theorem \ref{thm.main}. We want to mention three questions that came up naturally during the work on this article that we want to address in the future.
\begin{enumerate}
	\item[(1)] Under what assumptions can we transport all of these results to motives defined via Spitzweck's motivic cohomology ring spectrum?
	\item[(2)] Is it enough to invert the residue characteristics of the base for all of our results?
	\item[(3)] Can the results of this article be extended to other cohomology theories?
\end{enumerate}

Let us go a bit more into detail. So, let $S=\Spec(k)$ be the spectrum of a field and $X$ be an $S$-scheme locally of finite type. further, let $G$ be a split reductive group over $k$ with split maximal torus $T$ and associated Weyl group $W$. Assume $G$ acts on $X$. \\ \par 

Question (1) is rather straightforward. For Chow groups one can see that if $A^{\bullet}_{G}(X)\cong A^{\bullet}_{T}(X)^{W}$ holds integrally if any only if $G$ is special. But if we assume that $G$ is special, then any $G$-torsor is trivial \textit{Zariski}-locally. If we define integral motives $\DM_{\ZZ}$ on Prestacks via right Kan extension of Spitzweck motives (cf. \cite{Spitz}), we see that $\DM_{\ZZ}$ satisfies Nisnevich and in particular Zariski descent. Thus, up to technicalities, we expect that all the arguments after Section \ref{sec.N} go through. Crucially, Section \ref{sec.fin} has to be worked out in this context, as both Ayoub and Cisinski-D\'eglise use rational coefficients. This is not surprising, as one can see that the particular motivic behavior of torsors under finite groups should yield \'etale descent. We hope that in the case of special group there is a work around.\\ \par

Question (2) is addressed in a similar fashion, instead of Spitzweck motives, we can use \'etale motive (cf. \cite{etMot}). As these still satisfy \'etale descent again all the arguments after Section \ref{sec.N} should go through. We expect that inverting the residue characteristics should be enough to recover the statements of Section \ref{sec.fin} in this case. But still one needs to prove the necessary results, which we expect to be rather straightforward.\\ \par

Question (3) needs more careful treatment. The results about $T$-torsors of Section \ref{sec.T-tors} can be extended to other oriented cohomology theories, as we have seen. Section \ref{sec.N} is more difficult as it boils down to vanishing results on cohomology theories. The key part is Proposition \ref{prop.1.BG}. We expect that this proposition still holds for modules over the \'etale $K$-theory ring spectrum but we did not check this thoroughly. If one wants to work with genuine $K$-theory, then this statement can not be proven via \'etale descent and thus needs a more careful treatment. For other cohomology theories one can possibly give a precise vanishing assumption to extend the results of this article.\par 
For completion let us note, that the computation of the global sections in Theorem \ref{thm.motive.BG} will not work for arbitrary cohomology theories. If one takes for example Chow-Witt cohomology and $G=\SL_{2}$, then the Chow-Witt ring of $\B{\GG_{\textup{m}}}$ is trivial, while the Chow-Witt ring of $\B{G}$ is not (cf. \cite{LevineEx}).

\begin{appendix}

\section{Existence of maximal tori over Pr\"ufer domains}
\label{sec.app}
The following is essentially from the authors notes of a seminar talk from Torsten Wedhorn and we do not claim originality (cf.  \cite{Tors}). I want to thank him for allowing me to use his ideas.
\par 
For our main result (Theorem \ref{thm.main}), we need to work with reductive group schemes that admit maximal tori. As we mentioned in Section \ref{sec.gen} not every reductive group scheme admits such a maximal tori. In this short appendix, we want to prove that a reductive group scheme over sufficiently nice base, e.g. Dedekind domains, with quasi-split generic fibre admits a maximal torus. In fact, to prove this, we only need to assumptions on the base. First, we will need that the base is affine and secondly the stalks of the base are valuation rings. The following idea was communicated to us by Torsten Wedhorn. \par
Rings with the property that the stalks are valuation rings were studied before and we will give them a name according to literature.

\begin{defi}[\protect{\cite{Pruf}}]
	An integral domain $A$ is called \textit{Pr\"ufer domain}, if for any prime $\pfr\subseteq A$, the localization $A_{\pfr}$ is a valuation ring.
\end{defi}

\begin{example}
\label{ex.app}
	Any Dedekind domain is a Pr\"ufer domain conversely, any noetherian Pr\"ufer domain is in fact a Dedekind domain (cf. \cite[00II]{stacks-project}).
	So, Pr\"ufer domains can be seen as the non-noetherian analog of Dedekind domains.\par 
	Another example of Pr\"ufer domains, are so called B\'ezout domains, i.e. integral domains in which every finitely generated ideal is principal. In particular, the Picard group of a B\'ezout domain is trivial. This includes all PID's and valuation rings, but also rings such as the algebraic integers $\overline{\ZZ}$ (cf. \cite[pp. 243-245]{Bez} for more examples).
\end{example}

In the following let us fix a Pr\"ufer domain $A$ and set $S\coloneqq \Spec(A)$.

\begin{lem}
\label{lem.app}
	Let us denote $K\coloneqq \Quot(A)$. Further, let  $X\rightarrow S$ be a proper morphism of finite presentation. Then the map $X(A)\rightarrow X(K)$ is bijective.
\end{lem}
\begin{proof}
	Let $\sigma\in X(K)$. By the valuative criterion for properness, we can find for any $s\in S$ a unique lift $\sigma_{s}\in X(\Ocal_{S,s})$ of $\sigma$. As $X$ is of finite presentation, we can find an open neighborhood $U_{s}$ of $s$ and a lift $\sigma_{U_{s}}\in X(U_{s})$ of $\sigma_{s}$. By uniqueness of the lifts $\sigma_{s}$, the $\sigma_{U_{s}}$ glue to a unique lift $\sigma_{A}\in X(A)$ of $\sigma$.
\end{proof}

Applying this lemma to the scheme representing Borels in a reductive $S$-group scheme $G$, we see that if the generic fiber of $G$ is quasi-split, then $G$ is quasi-split. The existence of a maximal torus in $G$ follow now from the fact maximal tori contained in $B$ are a $R_{u}(B)$-torsor. Writing everything out we get the following proposition.

\begin{prop}
\label{prop.app}
	Let $G$ be a reductive $S$-group scheme. If the generic fibre of $G$ is quasi-split, then $G$ admits a Borel pair. Further, if $\Pic(A)=0$ and the generic fibre of $G$ is split, then $G$ is split.
\end{prop}
\begin{proof}
	Let $K\coloneqq \Quot(A)$ and $\Bor$ be the scheme classifying the Borels of $G$ over $A$. As $G_{K}$ is quasi-split, there exists an element $B_{K}\in \Bor(K)$. The scheme $\Bor$ is proper and of finite presentation (cf. \cite[Exp. XXII Cor. 5.8.3]{SGA3}) and thus we can apply Lemma \ref{lem.app} to $\Bor$ and see that $B_{K}$ lifts uniquely to a Borel $B$ of $G$. Let $U$ denote the unipotent radical of $B$. The functor of maximal tori in $B$ is representable by a smooth affine $S$-scheme $X$ and is further a $U$-torsor (cf. \cite[Exp. XXII Cor. 5.6.13]{SGA3}). But $H^{1}(S,U)=0$, as $S$ is affine (cf. \cite[Exp. XXII Cor. 5.9.6]{SGA3}), and therefore $X\rightarrow S$ admits a section. \par 
	If further $T_{K}$ is split and $\Pic(A)=0$, then we claim that $G$ is split. For this let $\hgbar$ be the geometric point corresponding to $K$. As $S$ is integral and normal (cf. \cite[00IC]{stacks-project}), $T$ splits after passage to a finite Galois cover of $S$ (cf. \cite[Cor. B.3.6]{Conrad}). This induces an action of $\pi_{1}^{\et}(A,\hgbar)$ on the character group $X^{\bullet}(T)$, which is trivial if and only if $T$ is split. As $T_{K}$ is split, the action of $\pi_{1}^{\et}(K,\hgbar)$ on $X^{\bullet}(T_{K})$ is trivial. Since $\pi_{1}^{\et}(A,\hgbar)$ is a quotient of $\pi_{1}^{\et}(K,\hgbar)$ (cf. \cite[0BQM]{stacks-project}), we deduce that $\pi_{1}^{\et}(A,\hgbar)$ acts trivially on $X^{\bullet}(T)$. Therefore, $T$ is split and since $\Pic(A)=0$, we see that $G$ is split (cf. \cite[Exp. XXII Prop. 2.2]{SGA3}).
\end{proof}

\end{appendix}

\addcontentsline{toc}{section}{References}
\bibliographystyle{halpha-abbrv}
\bibliography{motiveG}

\end{document}